\documentclass[11pt]{amsart}
\usepackage{a4wide}

\usepackage[utf8]{inputenc}
\usepackage{amsmath}
\usepackage{amsfonts}
\usepackage{amssymb}
\usepackage{amsthm}
\usepackage{esint}
\usepackage{bbm}
\usepackage{cancel}
\usepackage{color, comment}
\usepackage[square,sort,comma,numbers]{natbib}
\usepackage[colorlinks,citecolor=blue]{hyperref}
\usepackage{tikz}
\usepackage{subcaption}
\usepackage{listofitems}
\usepackage{pgfplots}
\usepackage{mathrsfs}
\usetikzlibrary{arrows}
\usepackage[a4paper,left=35mm,right=35mm,top=34mm,bottom=40mm]{geometry} 

\definecolor{ccqqqq}{rgb}{0.8,0.,0.}
\definecolor{qqqqff}{rgb}{0.,0.,1.}
\definecolor{xfqqff}{rgb}{0.4980392156862745,0.,1.}

\usepackage{import}
\usepackage{xifthen}
\usepackage{pdfpages}
\usepackage{transparent}

\numberwithin{equation}{section}
\newcommand{\newatop}{\genfrac{}{}{0pt}{1}}
\newcommand{\ud}{\,\mathrm{d}}

\newcommand{\Hcal}{\mathcal{H}}
\newcommand{\Ccal}{\mathcal{C}}

\newcommand{\E}{\mathcal{E}}
\newcommand{\F}{\mathcal{F}}
\newcommand{\Fsf}{\mathsf{F}}
\newcommand{\G}{\mathsf{G}}

\newcommand{\N}{\mathbb{N}}

\newcommand{\R}{\mathbb{R}}
\newcommand{\Z}{\mathbb{Z}}

\newcommand{\A}{\mathcal{A}}
\newcommand{\Fa}{\mathsf{F}}
\newcommand{\Int}{\mathrm{int}}
\newcommand{\ext}{\mathrm{ext}}
\newcommand{\Wi}{\mathrm{wire}}
\newcommand{\Per}{\mathrm{P}}
\newcommand{\II}{\mathsf{I}}
\newcommand{\Ii}{I}
\newcommand{\planar}{\mathrm{planar}}
\newcommand{\clos}{\mathrm{clos}}

\newcommand{\ubd}{\mathrm{ubd}}
\newcommand{\bdd}{\mathrm{bdd}}
\newcommand{\TT}{\mathfrak T}
\newcommand*{\genbf}[1]{\ifmmode\mathbf{#1}\else\textbf{#1}\fi}

\newcommand{\Fill}{\mathrm{Fill}}

\newcommand{\Ed}{\mathsf{Ed}}
\newcommand{\comb}{\mathrm{comb}}

\usepackage{pict2e}
\makeatletter
\DeclareRobustCommand{\intprod}{%
  \mathbin{\mathpalette\int@prod{(0.1,0)(0.9,0)(0.9,0.8)}}}
\DeclareRobustCommand{\restrict}{%
  \mathbin{\mathpalette\int@prod{(0.1,0.7)(0.1,0.1)(0.7,0.1)}}\!}	
\newcommand{\int@prod}[2]{%
  \begingroup
  \sbox\z@{$\m@th#1+$}%
  \setlength\unitlength{\wd\z@}%
  \begin{picture}(1,1)
  \roundcap
  \polyline#2
  \end{picture}%
  \endgroup
}
\makeatother

\newtheorem{theorem}{Theorem}
\newtheorem{lemma}[theorem]{Lemma}
\newtheorem{proposition}[theorem]{Proposition}

\theoremstyle{definition}
\newtheorem{definition}[theorem]{Definition}
\newtheorem{remark}[theorem]{Remark}

\date{\today}

\author[G. Del Nin]
{Giacomo Del Nin} 
\address[Giacomo Del Nin]{MPI MiS, Inselstrasse 22, 04103 Leipzig, Germany
}
\email[G. Del Nin]{giacomo.delnin@mis.mpg.de}

\author[L. De Luca]
{Lucia De Luca}
\address[Lucia De Luca]{
IAC-CNR, Via dei Taurini, 19 I-00185 Rome, Italy
}
\email[L. De Luca]{lucia.deluca@cnr.it}

\title[The square sticky disk]{The square sticky disk:\\ crystallization and Gamma-convergence\\ to the octagonal anisotropic perimeter}

\begin{document}
\begin{abstract}
    We consider a variant of the sticky disk energy where distances between particles are evaluated through the sup
  norm $  \lVert\cdot\rVert_{\infty}$
     in the plane. We first prove crystallization of minimizers in the square lattice, for any fixed number $N$ of particles. Then we consider the limit as $N\to \infty$: in contrast to the standard sticky disk, there is only one orientation in the limit, and we are able to compute explicitly the $\Gamma$-limit to be an anisotropic perimeter with octagonal Wulff shape. The results are based on an energy decomposition for graphs that generalizes the one proved by De Luca-Friesecke \cite{DeLucaFriesecke} in the triangular case.
     \vskip5pt
	\noindent
	\textsc{Keywords:} Crystallization; graph theory; sticky disk; variational methods; Gamma-convergence; anisotropy.
	\vskip5pt
	\noindent
	\textsc{MSC class:} 49J45 (primary); 05C10, 70C20, 82D25 (secondary).
\end{abstract}
    \maketitle 

\section{Introduction}

The \textit{sticky disk} is a very simple model of interacting particles in the plane. Given a (finite) configuration of distinct points $X=\{x_1,\ldots,x_N\}\subset \R^2$ the sticky disk energy is defined:
\begin{equation}\label{eq:energy_def}
    \E_{\mathrm{HR}}(X):=\frac 1 2\sum_{\newatop{1\le i,j\le N}{i\neq j}} V(|x_i-x_j|).
\end{equation}
Here, the symbol $|\cdot|$ denotes the Euclidean norm and $V:(0,\infty)\to [-1,+\infty]$ is the sticky disk (or \textit{Heitmann-Radin}) potential
\begin{equation}\label{eq:sticky_disk_potential}
    V(r):=
    \begin{cases}
        +\infty & \text{if $0<r<1$}\\
        -1 & \text{if $r=1$}\\
        0 & \text{if $r>1$}.
    \end{cases}
\end{equation}
This model was introduced and studied in 1980 by Heitmann and Radin \cite{HeitmannRadin}, who proved (building upon previous work by Harborth \cite{Harborth}) that for a fixed $N$ minimizers of the energy $\E_{\mathrm{HR}}$ are crystallized, i.e., they form a subset of the unit-step triangular lattice. An alternative proof based on an energy decomposition for graphs was given more recently in \cite{DeLucaFriesecke}. Radin later proved that crystallization holds also for the slightly less singular \textit{soft disk} potential \cite{Radin81}, and also in this case the energy decomposition from \cite{DeLucaFriesecke} was recently used by the authors to prove a sharp version of this result \cite{DelNinDeLuca}.

In the present work we analyze a specific anisotropic variant of the sticky disk model, where distances are evaluated by means of the norm $\|\cdot\_\infty$ (also called \textit{square norm}) in the plane:
\begin{equation}\label{introen}
    \E(X):=\frac 1 2\sum_{\newatop{1\le i,j\le N}{i\neq j}} V(\|x_i-x_j\|_\infty).
\end{equation}
Such a functional has been first studied in \cite{Brass} where it has been shown that for every $N\in\N$
$$
\min_{\sharp X=N}\E(X)=\min_{\newatop{\sharp X=N}{X\subset \Z^2}}\E(X),
$$
namely that there always exists a minimum that is a subset of $\Z^2$. This suggests a finite crystallization result in the square lattice $\Z^2$, namely that \textit{every} minimum is a subset of a copy of $\Z^2$.
In Section \ref{sec:crystallization}, we prove this property (called \textit{finite crystallization}), for $N\ge 6$, for the energy $\E$.
Furthermore, in Section \ref{sec:gamma_convergence}, we develop also an asymptotic analysis (as the number $N$ of particles diverges) for quasi-minimizers, in the spirit of what was done in \cite{DeLucaNovagaPonsiglione, FriedrichKreutzSchmidt} for the Euclidean norm, that is, for the regular triangular lattice.

Let us now describe our results a bit more in detail. It is not hard to see that the maximum number of interactions for every particle is 8. Hence, ground states of the energy $\E$ should have as many points with 8 neighbors as possible.
 Such information, together with the rigidity of the interaction potential and of the norm $\|\cdot\|_\infty$, allows to show that in optimal configurations most of the particles lie on square crystallized faces, which, in turns, yields that all the faces are actually crystallized: they are either squares or triangles, corresponding to half squares.

By construction, any of the ``inner'' particles (i.e., surrounded by $8$ neighbors) in an optimal configuration gives an energy contribution equal to $-4$; therefore, the leading order of the minimal energy is $-4N$. The remaining particles are thus ``boundary'' particles (surrounded by less than $8$ neighbors) and their number is of the order of a perimeter, i.e., of $\sqrt N$ (being $N$ the ``volume'').
In other words, it can be shown that
\[
-4N+c_1\sqrt{N}\le \min_{\sharp X=N}\E({X})\le -4N+c_2\sqrt{N}
\]
for some positive constants $c_1,c_2$. It is thus natural to check whether (local or global) crystallization takes place in the limit as $N\to\infty$ also for \textit{low-energy sequences}, namely to sequences $\{X_N\}_{N\in\mathbb{N}}$ with $\# X_N=N$ and satisfying 
\begin{equation}\label{intro:bounden}
\E(X_N)\leq -4N+C\sqrt{N}
\end{equation}
for some constant $C$. 
To this end,
we associate with every (finite-energy) configuration $X$ the finite perimeter set
\begin{equation}\label{eq:A(X)_definition}
A_\boxtimes(X):=\bigcup_{F\in \Fsf_\boxtimes(X)} F,
\end{equation}
where $\Fsf_\boxtimes(X)$ denotes the family of all the crystallized square faces of $X$. 
In Proposition \ref{prop:compactness}, we provide a compactness result for the energy $\E$, which states that if a sequence $\{X_N\}_{N\in\N}$ satisfies \eqref{intro:bounden} then, up to a subsequence, the sets $N^{-1/2}A_\boxtimes(X_N)$ converge to a finite perimeter set $E$. 
In view of the rigidity highlighted above, 
the limit set $E$ presents a single orientation; in other words, sequences satisfying \eqref{intro:bounden} still exhibit asymptotic crystallization. Such a behavior is somehow in contrast with the isotropic case of the Euclidean norm, where the crystallization is only ``local", since polycrsytalline structures may arise \cite{DeLucaNovagaPonsiglione,FriedrichKreutzSchmidt}.

Finally, as one may expect, the compactness theorem stated above is accompanied with a $\Gamma$-convergence result (Theorem \ref{thm:gamma_convergence}).
Such a result shows that the $\Gamma$-limit of the energy $\E$, once renormalized with the volume term $-4N$ and scaled by the length factor $N^{1/2}$, is the anisotropic perimeter  
\begin{equation}\label{eq:anisotropic_perimeter}
\Per_\phi(E):=\int_{\partial^* E} \phi(\nu_E)\,d\Hcal^{1}
\end{equation}
associated with the anisotropy $\phi$ defined by
\begin{equation}\label{finsler}
\phi(\nu_1,\nu_2):=|\nu_1|+|\nu_2|+|\nu_1+\nu_2|+|\nu_1-\nu_2|.
\end{equation}
The construction of the recovery sequence is quite standard (see Section \ref{sec:gamma_liminf}), hence we only briefly describe the strategy for the proof of the lower bound.
Despite the rigidity mentioned above, low energy sequences do not exhibit exact crystallization (i.e., they are not subsets of a square lattice), hence the proof of the $\Gamma$-convergence still requires some fine combinatorial arguments and is thus based on an energy decomposition for graphs that generalizes the one proved in \cite{DeLucaFriesecke}; such a decomposition is derived in  Section \ref{sec:energy_decomposition} and is used also in Section \ref{sec:crystallization} to prove finite crystallization. 

 As customary, we associate to a finite energy configuration $X=\{x_1,\ldots,x_N\}$ its \textit{bond graph} $\G(X)$, whose vertices are the points of $X$, and where edges are considered between any pair of points $x,x'\in X$ with $\|x-x'\|_\infty=1$. In order to make the presentation self-contained, in Section \ref{sec:preliminaries} we have included a (not too) short recap on graph theory.
 
 Contrary to the isotropic sticky disk, the procedure described above does not necessarily define a planar graph. However, the only source of non-planarity comes from square faces having both diagonals, that are actually the faces that one wants to observe when proving crystallization. 

 Recalling that the {\it ($\G$)-degree} of a point $x\in X$ is defined as
$\deg(x)=\deg_{\G}(x):=\sharp\{\{x,y\}\in\Ed(X)\}$, the energy decomposition, stated  in Theorem \ref{thm:decomposition_square}, allows us to rewrite the excess energy 
\begin{equation}\label{eq:excess_energy_def}
\F(X):=2\E(X)+8\sharp X=\sum_{x\in X} (8-\deg(x))
\end{equation}
as a sum of different contributions. 
 The most important ones are a combinatorial perimeter-type term, and a contribution due to the \textit{angular defect} of a face.
 The latter is a notion that we introduce here and that measures in some sense how far is a face from being crystallized. The main feature of our energy decomposition is its local character, which makes its application very powerful.
Indeed, with such a tool, in order to prove the lower bound we select a family of faces whose union has the same limit of $A_{\boxtimes}(X_N)$ and whose crystalline perimeter $\Per_\phi$ behaves, up to an asymptotically vanishing error, as the energy of $X_N$.
We can thus appeal to the lower semicontinuity of $\Per_\phi$ to deduce the sought liminf inequality.\\

In conclusion, some comments are in order. 
In this paper, we have proven finite (and ``infinite'') crystallization in $\Z^2$ for the Heitmann-Radin potential acting through the  distance induced by the norm $\|\cdot\|_\infty$ between the particles. To the best of our knowledge, this is the first finite crystallization result in the square lattice for a pairwise interaction energy. A finite crystallization result in the square lattice, for a functional combining a pairwise potential with a three-body interaction, has been obtained in \cite{MaininiPiovanoStefanelli}. Despite the absence of the three-body term, our setting is sensitively more rigid than that treated in \cite{MaininiPiovanoStefanelli}, as the orientation of the square faces in our case is uniquely determined. For what concerns the pairwise interactions and the square lattice, another crystallization result in the sense of thermodynamic limit has been obtained in \cite{BeterminDeLucaPetrache}. In that case the chosen norm is the  Euclidean one and the interaction potential is given by
\begin{equation}\label{llm}
V_{\mathrm{square}}(r):=\begin{cases}
   +\infty & \text{if $0<r<1$}\\
   -1 & \text{if $1\le r\le \sqrt{2}$}\\
   0 & \text{if $r>\sqrt{2}$.}
\end{cases}
\end{equation}
We hope that some of the combinatorial tools developed here could be useful to treat the finite crystallization problem in this more general setting.

As commented in Remark \ref{damettere}, our decomposition of the energy is purely combinatorial and applies to other potentials for which crystallization in the lattice $\Z^2$ is expected, such as that considered in \cite{BeterminDeLucaPetrache}). Moreover, in Subsection \ref{decorettri}, we show that a similar decomposition applies also to the case of triangular lattice.

The Wulff shape for the $\Gamma$-limit $\Per_\phi$ of the energy $N^{-1/2}(\E+4N)$ is the regular (with respect to the norm $\|\cdot\|_\infty$) octagon; therefore, in view of our $\Gamma$-convergence analysis, minimizers of $\E$ converge to the regular octagon.
It is a natural question to investigate the values of $N$ for which minimizers are uniquely determined (as done in \cite{DeLucaFriesecke2} for the triangular lattice) as well as to check if the so-called $N^{3/4}$ law holds true also in this case (see \cite{Schmidt, MaininiPiovanoSchmidtStefanelli, DavoliPiovanoStefanelli, CicaleseLeonardi, CicaleseKreutzLeonardi} for proofs of this fact in several contexts).

Finally, having studied the cases of the Euclidean norm and of the square norm $\|\cdot\|_\infty$, we could ask what happens for general norms in the plane.
In this respect, the recent paper \cite{BeterminFurlanetto} provides interesting insights on the problem in the case of general $\|\cdot\|_p$ norms.
We plan to investigate this issue in the future; in a nuthsell, for general strictly convex norms, minimizers crystallize in a triangular lattice (see \cite{Brass}) but for non-strictly convex ones they may ``choose'' between different crystalline structures. 
This aspect makes the analysis of quasi-minimizers more intriguing, since in such a case polycrystalline structures may arise as a mixture of different lattices.

\subsection*{Acknowledgements:} LDL is member of the Gruppo Nazionale per l'Analisi Matematica, la Probabilit\`a e le loro Applicazioni (GNAMPA) of the Istituto Nazionale di Alta Matematica (INdAM).
LDL acknowledges the financial support of PRIN 2022HKBF5C ``Variational Analysis of
complex systems in Materials Science, Physics and Biology'', PNRR Italia Domani, funded
by the European Union via the program NextGenerationEU, CUP B53D23009290006.

Views and opinions expressed are however those of the authors only and do not necessarily
reflect those of the European Union or The European Research Executive Agency. Neither
the European Union nor the granting authority can be held responsible for them.

\section*{Notation}

{\renewcommand{\arraystretch}{1.2}
\begin{tabular}{c|l|c}
    \textbf{Symbol} & \textbf{Meaning} & \textbf{Appearence}\\\hline
    $X$ & Set of points (particles) & Intro\\
    $\E(X)$ & Total energy & \eqref{introen} \\
    $\F(X)$ & Excess energy: $\F(X)=\E(X)+8N$ & \eqref{eq:excess_energy_def} \\
    $\mathsf{E}(X)$ & Set of \textit{all} edges of $X$ & Sec. \ref{sec:preliminaries} \\
    $\G$ & Graph & Sec. \ref{sec:preliminaries} \\
    $\Ed$ & Set of edges of a specific graph (subset of $\mathsf{E}$) & Sec. \ref{sec:preliminaries} \\
    $I(\G)$ & Union of all segments associated with its edges & Sec. \ref{sec:preliminaries} \\
    $\Fsf(\G)$ & Set of all faces of $\G$ & Sec. \ref{sec:preliminaries} \\
    $\Fsf_\boxtimes(\G)$ & Family of non planar faces under condition \eqref{crossass},& \\
    &i.e., family of quadrilateral faces having both diagonals & Sec. \ref{sec:preliminaries} \\
    $\Fsf_\mathrm{planar}(\G)$ & Family of planar faces of $\G$ & Sec. \ref{sec:preliminaries} \\
    $F^{\mathrm{ubd}}(\G)$ & Unbounded face of $\G$ & Sec. \ref{sec:preliminaries} \\
    $\Fsf_{\mathrm{bdd}}$ & Family of bounded faces of $\G$ & Sec. \ref{sec:preliminaries} \\
    $\Ed^{\mathrm{int}}$ & Interior edges & Sec. \ref{sec:preliminaries}  \\
    $\Ed^{\mathrm{wire,ext}}$ & Exterior wire edges & Sec. \ref{sec:preliminaries} \\
    $\Ed^{\mathrm{wire,int}}$ & Interior wire edges & Sec. \ref{sec:preliminaries} \\
    $\Ed^\partial$ & Boundary edges & Sec. \ref{sec:preliminaries} \\
    $\Ed^\partial_S(\G)$ & Boundary edges associated to $A_S(\G)$ & \eqref{eq:Ed_partial_S}\\
    $\Ed^\ext_S(\G)$ & Exterior edges associated to $A_S(\G)$ & \eqref{eq:Ed_ext_S} \\
    $\partial X$ & Boundary of $X$ & \eqref{eq:partial_X} \\
    $\Per_\comb(\G)$ & Combinatorial perimeter of $\G$, equal to $\# \Ed^\partial+2\#\Ed^{\mathrm{wire,ext}}$ & \eqref{percomb} \\
    $A(\G)$ & Union of all bounded faces of $\G$ & \eqref{defunifacce} \\
    $A_S(\G)$ & Union of all bounded faces appearing in $S$ & \eqref{defAS} \\
    $A_\boxtimes(X)$ & Union of all crystallized square faces of $\G$  & \eqref{eq:A(X)_definition}\\
    $\Ccal(T)$ & Family of connected components of a closed set $T$ & Subsec. \ref{subsec:components} \\
    $U(T)$ & Unbounded connected component of $\R^2\setminus T$, for compact $T$ & Subsec. \ref{subsec:components}  \\
    $\Fill(T)$ & Filling of $T$, given by $\R^2\setminus \overline{U}(T)$ & Subsec. \ref{subsec:components}   \\
    $C^\ext(\partial T)$ & Exterior component of $T$ & Subsec. \ref{subsec:components} \\
    $\Ccal^{\mathrm{int}}(\partial F)$ & Family of interior components of $T$ & Subsec. \ref{subsec:components}  \\
    $\delta(\alpha)$ & Angular defect: $\frac{4}{\pi}\alpha-1$ & \eqref{eq:angular_defect} \\
    $\delta_\Delta(\alpha)$ & Angular defect for the triangular case: $\frac{3}{\pi}\alpha-1$ & \eqref{eq:triangular_defect} \\
    $\delta(F)$ & Defect of a face (sum of the defects of its internal angles) & \eqref{eq:delta_F}\\
    $\delta_\Delta(F)$ & Same as above but for the triangular case & \eqref{eq:triangular_defect} \\
    $\A,\A_N$ & Admissible configurations with/without a cardinality constraint & \\
    $\A^{\Z^2},\A^{\Z^2}_N$ & and with/without the constraint of being subsets of $\Z^2$ & Sec. \ref{sec:crystallization} \\

    $\phi$ & Anisotropy given by $\phi(\nu)=|\nu_1|+|\nu_2|+|\nu_1+\nu_2|+|\nu_1-\nu_2|$ & \eqref{finsler} \\
    $\Per_\phi$ & Anisotropic perimeter & \eqref{eq:anisotropic_perimeter} \\
    $\ell_\phi[e]$ & $\phi$-length of a segment $e$: $\Hcal^1(e)\phi(\tfrac{e^\perp}{|e^\perp|})$ & Subsec. \ref{sec:gamma_liminf} \\

\end{tabular}
}

\section{Preliminaries on planar graphs}\label{sec:preliminaries}

Here we collect some notions and notation on graphs that will be adopted in this paper. 
Let $X$ be a finite subset of $\R^2$ and let $\Ed$ be a given subset of $\mathsf{E}(X)$, where
\begin{equation*}
\mathsf{E}(X):= \{\{x,y\}\subset \R^2\,:\, x,y\in X\,,\,x\neq y\}.
\end{equation*}
The pair $\G=(X,\Ed)$ is called {\it graph};  $X$ is called  the set of {\it vertices} of $\G$ and $\Ed$ is called the set of {\it edges} (or {\it bonds}) of $\G$.

For any graph $\G$ we set 
\begin{equation*}
\II(\G):=\{[x,y]\,:\,\{x,y\}\in\Ed\}\qquad\textrm{and}\qquad \Ii(\G):=\bigcup_{\{x,y\}\in\Ed}[x,y].
\end{equation*}
In what follows, for every $e=\{x,y\}\in\Ed$ we will denote by $[e]$ the (closed) segment $[x,y]$ belonging to $\II(\G)$.

Given $X'\subset X$, we denote by $\G_{X'}$ the {\it subgraph} (or {\it restriction}) of $\G$ generated by $X'$\,, defined by $\G_{X'}=(X',\Ed')$ where $\Ed':=\{\{x',y'\}\in\Ed\,:\, x',y'\in X'\}$.

\begin{definition}\label{conncomp}
Let $\G=(X,\Ed)$ be a graph.
We say that two points $x,z\in X$ are connected, and we write $x\sim z$, if there exist a number $M\in\N$ and a {\it path} $x=y_0,\ldots,y_M=z$ such that $\{y_{m-1},y_m\}\in\Ed$ for every $m=1,\ldots,M-1$.

We say that  $\G_{X_1},\ldots,\G_{X_K}$ with $K\in\N$ are the {\it connected components} of $\G$ if 
$\{X_1,\ldots,X_{K}\}$ is a partition of $X$
and for every $k,k'\in\{1,\ldots,K\}$ with $k\neq k'$ it holds
\begin{align*}
x_k\sim y_k\qquad&\textrm{for every }x_k,y_k\in X_k\,,\\
x_{k}\not\sim x_{k'}\qquad&\textrm{for every }x_k\in X_k\,, x_{k'}\in X_{k'}.
\end{align*}
If $\G$ has only one connected component we say that $\G$ is {\it connected}.
\end{definition}
Let $\G=(X,\Ed)$ be a connected graph.
We say that $\G$ is planar if closed edges only intersect at vertices.
For every $x\in X$ we denote by $\deg(x)=\deg_{\G}(x)$ the {\it degree} of $x$ (in $\G$), i.e., the cardinality of bonds in $\Ed$ that contain $x$.
\begin{figure}[htbp]
    \centering
    \includegraphics[scale=0.8]{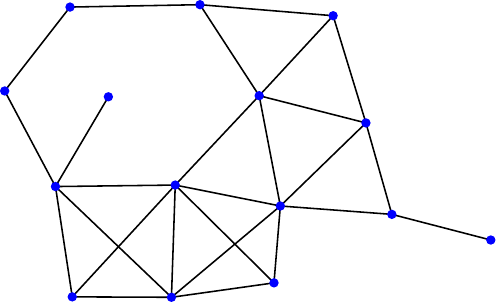}
    \caption{An admissible graph: the graph has straight edges and is planar, with the only possible exception of convex quadrilateral faces having both diagonals.}
    \label{fig:graph}
\end{figure}

Throughout the paper we will focus on graphs $\G=(X,\Ed)$ satisfying the following properties
\begin{align}\label{maxdeg}\tag{{MD}}
&
\deg(x)\le 8\qquad\textrm{for any }x\in X;\\ 
\label{crossass}\tag{{CE}}
&\textrm{if }\{x_1,x_3\},\{x_2,x_4\}\in\Ed\textrm{ and }\#((x_1,x_3)\cap(x_2,x_4))=1\,,\\ \nonumber
&\textrm{ then }\{x_j,x_k\}\in\Ed\textrm{ for every }j,k=1,\ldots,4\textrm{ with }j\neq k\,.
\end{align}
In other words, the assumption \eqref{maxdeg} (\textit{maximum degree}) provides an upper bound for the degree of each vertex, whereas the assumption \eqref{crossass} (\textit{crossing edges}) states that the only admissible non-planarity comes from (convex) quadrilateral faces having both diagonals.
In what follows we call {\it admissible} any graph satisfying the assumptions \eqref{maxdeg} and \eqref{crossass}.

We denote by $\Fa_{\boxtimes}(\G)$ the family of non-planar faces, i.e., those quadrilaterals whose vertices fulfill the condition expressed in \eqref{crossass}. 
By a {\it (finite) planar face}, we mean any open, bounded, connected component of $\R^2\setminus\Big(X\cup \Ii(\G)\cup \bigcup_{f\in\Fa_{\boxtimes}(\G)}\clos(f)\Big)$\,,
and we denote by $\Fa_{\planar}(\G)$ the class of finite planar faces of $\G$\,.
Finally, we denote by $F^{\ubd}(\G)$ the unbounded face, i.e., the (only) open unbounded connected component of $\R^2\setminus\Big(X\cup\Ii(\G)\Big)$.

We denote by 
$\Fa(\G):=\Fa_{\boxtimes}(\G)\cup\Fa_{\planar}(\G)\cup \{F^{\ubd}(\G)\}$\, the set of all faces of $\G$, and by $\Fsf_\bdd(\G):=F(\G)\setminus \{F^\ubd(\G)\}$ the set of bounded faces. In the following we might omit the dependency on $\G$ for notational simplicity.
We define the Euler characteristic of $\G$ as
\[
\chi(\G)=\# X-\#\Ed+\#\Fa(\G)\,.
\]




We set
\begin{equation}\label{defunifacce}
A(\G):=\bigcup_{F\in\Fa_{\bdd}(\G)}\clos(F)\,.
\end{equation}
With a little abuse of language we will say that an edge $\{x,y\}$ lies on a set $E\subset\R^2$ if the segment $[x,y]$ is contained in $E$\,.
We classify the edges in $\Ed$ in the following subclasses:
\begin{itemize}
\item $\Ed^{\Int}$ is the set of {\it interior edges}, i.e., of edges lying on the boundary of two (distinct) faces; 
\item $\Ed^{\Wi,\ext}$ is the set of {\it exterior wire edges}, i.e., of edges that do not lie on the boundary of any face;
\item $\Ed^{\Wi,\mathrm{int}}$ is the set of {\it interior wire edges}, i.e., of edges lying on the boundary of precisely one face but not on the boundary of its closure; 
\item $\Ed^{\partial}$  is the set of {\it boundary edges}, i.e., of edges lying on $\partial A(\G)$. 
\end{itemize}
We also define the set of \textit{wire edges} as $\Ed^{\Wi}:=\Ed^{\Wi,\mathrm{int}}\cup\Ed^{\Wi,\ext}$.
With little abuse of notation, we denote by $\partial X$ the set of {\it boundary} particles, i.e., 
\begin{equation}\label{eq:partial_X}
\partial X:=\{x\in X\,:\exists\, y\in X \text{ such that }\,\{x,y\}\in\Ed^{\partial}\cup \Ed^{\Wi,\ext}\}.
\end{equation}
We define the {\it combinatorial perimeter} of $\G$ as
\begin{equation}\label{percomb}
\Per_{\comb}(\G):=\#\Ed^{\partial}+2\#\Ed^{\Wi,\ext}\,.
\end{equation}
According with the definitions introduced above, if $A(\G)$ has simple and closed polygonal boundary and if $\#\Ed^{\Wi,\ext}=0$, then $\Per_\comb(\G)=\#\partial X$\,.

Furthermore, for any $S\subseteq \Fa_{\bdd}(\G)$\,, we define
\begin{equation}\label{defAS}
A_S(\G):=\bigcup_{F\in S}\clos(F).
\end{equation}
Notice that, if $S=\Fa_{\bdd}(\G)$, then $A_S(\G)\equiv A(\G)$.
We denote by
$\Per_{\comb}(A_S(\G))$ the combinatorial perimeter of $A_S(\G)$\,, that is, the number of boundary edges of $A_S(\G)$\,:
\[
\Per_\comb(A_S(\G)):=\#\{\{x,x'\}\in\Ed(\G):\,[x,x']\subseteq\partial A_S(\G)\},
\]
where $\partial A_S(\G)$ denotes the topological boundary of $A_S(\G)$.

Let $F$ be a planar face of $\G$. We denote by $\Ed(F)$ the set of edges of $F$, i.e., the set $\{\{x,y\}\in\Ed:\, [x,y]\subseteq \clos(F)\}$. Moreover, we split the class $\Ed(F)$ into two subfamilies: the family of \textit{interior wire edges} $\Ed^{\Wi,\Int}(F)$, i.e., those edges from $\Ed(F)$ that are not contained in the boundary of $\clos(F)$; the family of \textit{boundary edges},  $\Ed^\partial(F):=\Ed(F)\setminus\Ed^{\Wi,\Int}(F)$. We call \textit{combinatorial perimeter} of $F$ the quantity 
\[
\Per_\comb(F):=\# \Ed^\partial(F)+2\# \Ed^{\Int,\Wi}(F).
\]

\subsection{Components of a face}\label{subsec:components}
For any closed set $T\subset\R^2$, we denote by $\Ccal(T)$ the family of connected components of $T$.
If $T$ is a closed and bounded set, we let $U(T)$ be the only unbounded connected component of $\R^2\setminus T$ and we denote by $\Fill(T):=\R^2\setminus\overline{U}(T)$ the filling of $T$. 
For every bounded, closed and connected set $T$, we introduce the following concepts (see also Figure \ref{fig:delta_F}):
\begin{itemize}
    \item $C^{\ext}(\partial T)$ denotes the \textit{exterior component}
    of $\partial T$, i.e., the connected component of $\Ccal(\partial T)$ containing $\partial T\cap \partial U(T)$;
    \item $\Ccal^{\Int}(\partial F):=\Ccal(\partial T)\setminus \{C^\ext(\partial T)\}$ denotes the family of \textit{interior components} of $\partial T$.
\end{itemize}
We highlight that $C^\ext(\partial T)$ does not coincide with $\partial T\cap\partial U(T)$.
Trivially, if $T$ is either a tree or a point, then $T\equiv \partial T$ and $\Ccal(T)=\{C^{\ext}(T)\}$.
Finally, given an admissible graph $\G$, we set $\Ccal^{\ext}(\partial F^{\ubd}):=\emptyset$ and $\Ccal^{\Int}(\partial F^{\ubd}):=\Ccal(\partial F^{\ubd})$.


We record here the following lemma, which states that, given any selection of faces $S$, every interior component of an unselected face corresponds to a connected component of $A_S(\G)\cup \Ii(\G)\cup X$.
\begin{lemma}\label{lemma:double_counting}
    Let $\G=(X,\Ed)$ be an admissible graph, and let $\Fsf_\boxtimes(\G)\subseteq S\subseteq \Fsf_{\bdd}(\G)$ be a set of (bounded) faces of $\G$.
    Then, we have
   \begin{equation}\label{stessacard}
    \sum_{F\in\Fsf(\G)\setminus S} \# \Ccal^{\Int}(\partial F)=\# \Ccal(A_S(\G)\cup \II(\G)\cup X).
    \end{equation}
    
\end{lemma}
\begin{proof}
We first notice that, by construction,
any element $C\in\bigcup_{F\in\Fsf(\G)\setminus S} \Ccal^{\Int}(\partial F)$ is the exterior boundary of a connected component of $A_S(\G)\cup \II(\G)\cup X$. This proves the inequality ``$\le$'' in \eqref{stessacard}.

Analogously, for any connected component $E$ of $A_S(\G)\cup \II(\G)\cup X$ we have that $C^{\ext}(\partial E)\in \bigcup_{F\in\Fsf(\G)\setminus S} \Ccal^{\Int}(\partial F)$. 
This implies the inequality ``$\ge$'' in \eqref{stessacard} and concludes the proof of the lemma. 
\end{proof}

\section{Energy decomposition on graphs}\label{sec:energy_decomposition}

The aim of this section is to prove an energy decomposition for graphs arising as bond graphs for the square sticky disk. The statement is quite flexible and combinatorial in nature, and indeed holds under the mere assumption that ``topologically'' the graph looks like one arising from such a bond graph, namely satisfying \eqref{maxdeg} and \eqref{crossass} (see Fig. \ref{fig:graph}). Moreover, it can be adapted to other settings, and for this reason we also state a version for the standard sticky disk (that is, when the underlying ideal lattice is the triangular one). The latter is a generalization of a similar energy decomposition obtained in \cite{DeLucaFriesecke}. The main difference is that our decomposition can be in a sense localized, as one is allowed to select \textit{any} subset of faces, and gets a corresponding decomposition involving the perimeter of the union of only those faces (see Remark \ref{rmk:decomposition}). 
The proof is based on the notion of \textit{angular defect}, that we introduce here and that allows to somewhat simplify the original argument in \cite{DeLucaFriesecke}. Another difference is the following: in the original energy decomposition the defect of a face was defined as the number of additional edges needed to triangulate it; instead, here the defect of a face is the sum of all the angular defects at its vertices. These concepts coincide for a simply connected face, but they differ in general. 



\subsection{The graph}
Let $\G=(X,\Ed)$, with $X\subset\R^2$ finite and with straight edges, be a graph satisfying the admissibility properties \eqref{maxdeg} and \eqref{crossass}.
We define the \textit{excess energy} as
\begin{equation}\label{excess}
\F(\G):=\sum_{x\in X}\F(x),
\end{equation}
where $\F(x)$ is the local excess defined for every vertex $x$ to be
\[
\F(x):=8-\deg(x).
\]
Loosely speaking, the energy $\F$ counts the total number of \textit{missing bonds}, as compared to the ideal situation where each particle has 8 neighbors.

\subsection{Angular defects} Given an angle $\alpha$, we call
\begin{equation}\label{eq:angular_defect}
\delta(\alpha):=\tfrac{4}{\pi}\alpha-1
\end{equation}
the \textit{angular defect} of $\alpha$. This measures by how many multiples of $\tfrac{\pi}{4}$ the angle $\alpha$ exceeds $\tfrac{\pi}{4}$.
Now let $x\in X$. Let us call $\alpha_i(x)$, $i=1,\ldots,\deg(x)$, the angles between consecutive edges appearing around $x$ (all with a positive sign), and $\delta_i(x):=\delta(\alpha_i(x))$ the corresponding defects. 
We have the following simple but key identity:
\begin{equation}\label{eq:rewriting_F_defects}
\delta(x):=\sum_{i=1}^{\deg(x)} \delta_i(x)=\sum_{i=1}^{\deg(x)}\left(\frac{4}{\pi}\alpha_i(x)-1\right)=8-\deg(x) =\F(x).
\end{equation}
In other words, the excess energy $\F(x)$ coincides with the sum of the angular defects at $x$.

\begin{figure}[htbp]
    \centering
   \includegraphics[scale=0.3]{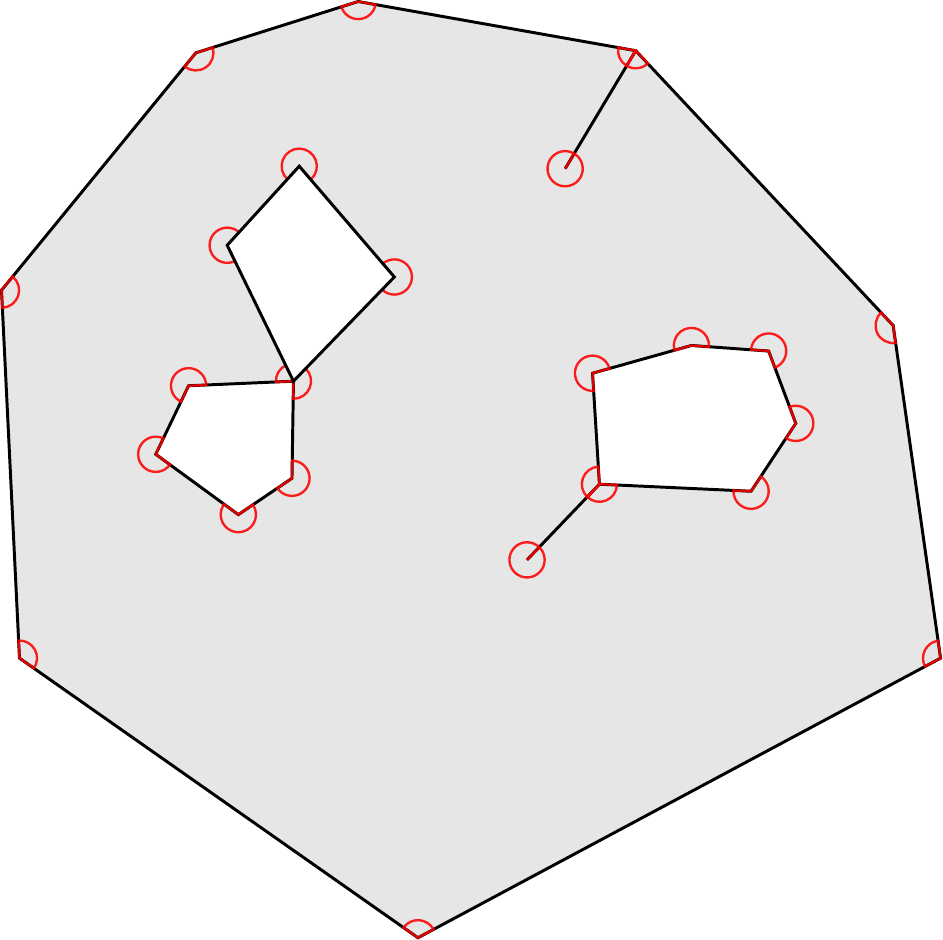}
    \caption{Example of a possible face $F$ (shaded in grey). The angles that contribute to $\delta(F)$ are highlighted in red. In this case $\delta(F)=3\Per_\comb(F)-8+8\#\Ccal^{\Int}( F)=3 \Per_\comb(F)+8$, and $\Per_\comb(F)=27$. Notice that in $\# \Ccal^\Int(F)$ we count the number of interior connected components of $\partial F$ (2 in this case), and not the number of interior connected components of $\R^2\setminus F$ (3 in this case).} 
    \label{fig:delta_F}
\end{figure}

\subsection{Defect of a face}
Given a simple polygon $P$ with $k$ sides, let us call $\beta_j$, $j=1,\ldots,k$ the internal angles (all considered to be positive). Then we define the \textit{defect} of $P$ as
\[
\delta(P):=\sum_{j=1}^k \delta(\beta_j).
\]
Since the sum of the internal angles of any polygon with $k$ sides is $(k-2)\pi$, we obtain
\[
\delta(P)=\frac{4}{\pi}(k-2)\pi-k=3k-8.
\]
We observe that, as a consequence, 
\[
\delta(P)=1\qquad\text{if $P$ is a triangle}.
\]
Now, given an admissible graph $\G=(X,\Ed)$, 
for every $F\in\Fsf_{\bdd}(\G)$ and for every $C\in\Ccal(\partial F)$, let $\gamma(C)$ be the curve describing $C$.
If $C$ is not made of only wire edges, we can orient $\gamma(C)$ in such a way that for every edge in $\partial F\setminus\Ed^{\Wi,\Int}(F)$ whose corresponding segment lies in the support of $\gamma(C)$, the curve $\gamma(C)$ finds $F$ on its right. This means that the orientation of $\gamma(C)$ is oriented clockwise if $C=C^{\ext}(\partial F)$ and counter-clockwise whenever $C\in\Ccal^{\Int}(\partial F)$ is not a tree; by relaxation, we can  
extend the counter-clockwise orientation also to the case where $C\in\Ccal^{\Int}(\partial F)$ is a tree, and hence to any $C\in\Ccal^{\Int}(\partial F)$. 

We now define the defect $\delta(F)$ of a face $F$ as the sum of the defects of all the angles facing the inside of $F$, and we will get an alternative expression in the next few lines.
For every $C\in\Ccal(\partial F)$ we denote by  $\ell(\gamma(C))$ the graph-length of the curve $\gamma(C)$ (that is, the number of boundary edges it is made of plus twice the number of wire edges); it follows that
$$
\sum_{\beta}\delta(\beta)=3\ell(\gamma(C))+8\epsilon(\gamma(C)),
$$
where the sum runs among all the angles between consecutive edges in $\Ed(F)$ facing the interior of $F$ (see Fig. \ref{fig:delta_F}) and $\epsilon(\gamma(C))$ represents the orientation of the curve $\gamma(C)$, that equals $-1$ if $C=C^{\ext}(\partial F)$ and  equals $+1$ otherwise.
 By summing over $C$ we can see that the defect of the face $F$ can be rewritten as
 \begin{align}
\begin{aligned}\label{eq:delta_F}
\delta(F)=&\sum_{C\in\Ccal(\partial F)} ( 3\ell(\gamma(C))+8\epsilon(\gamma(C)))\\
=&3\Per_{\comb}(F)+8(\#\Ccal^{\Int}(\partial F)-1),
\end{aligned}
\end{align}

Analogously, if $F$ is the unbounded face $F^{\ubd}$ there is no external connected component of $\partial F$, so that
\begin{equation}\label{eq:delta_F_unbounded}
\delta(F^{\ubd})=3\Per_\comb(F^{\ubd})+8\#\Ccal^{\Int}(F^{\ubd}).
\end{equation}
Note that $\#\Ccal^{\Int}(F^{\ubd})$ coincides with the cardinality of the connected components of 
\[
A_{\Fsf_\bdd(\G)}\cup I(\G)\cup X.
\]
In particular, $\delta(F)\ge 3\Per(F)-8\ge 1$ for any planar face. See also Figure \ref{fig:delta_F}.

Finally, we set $\delta(F):=0$ for every non-planar face, namely for $F\in\Fsf_\boxtimes(X)$. Observe that this is consistent with the sum of the internal angular defects at the 4 vertices. In conclusion, $\delta(F)$ is non-negative for \textit{every} face $F$, and at least 1 for every planar face.

\subsection{Energy decomposition}\label{subsec:energy_decomposition}
Let $\G=(X,\Ed)$ be an admissible graph. Let $S$ be any subset of the bounded faces of $\G$ that includes all the non-planar ones, namely, $\Fsf_\boxtimes(\G)\subseteq S\subseteq \Fsf_{\bdd}(\G)$, and let $A_S(\G)$ be the set defined in \eqref{defAS}.

Moreover, we define 
\begin{equation}\label{eq:Ed_partial_S}
\Ed^\partial_S(\G):=\{\{x,x'\}\in\Ed(\G)\,:\,[x,x']\subset \partial A_S(\G)\}
\end{equation}
as the family of edges lying on the boundary of $A_S(\G)$
and
\begin{equation}\label{eq:Ed_ext_S}
\Ed_S^{\ext}(\G):=\Ed(\G)\setminus\{\{x,x'\}\in\Ed(\G)\,:\,[x,x']\subset A_{S}(\G)\}
\end{equation}
as the family of edges of $X$ that intersect $A_S(\G)$ in at most 2 points (that is, those edges that on neither side touch a face from $S$).

Recall that $\F(\G)=\sum_{x\in X} \F(x)=\sum_{x\in X}(8-\deg_{\G}(x))$. We are now ready to state the central energy decomposition result.


\begin{theorem}[Energy decomposition for 8 neighbors]\label{thm:decomposition_square}
Let $\G=(X,\Ed)$ be a graph satisfying \eqref{crossass}.
Let $S$ be any subfamily of faces satisfying $\Fsf_\boxtimes(\G)\subseteq S\subseteq \Fsf_\bdd(\G)$. Then we can rewrite the total energy as
\begin{equation}\label{formuladeco}
\begin{aligned}
\F(\G)=&\,3\Per_\comb(A_S(\G))+8\#\Ccal(A_S(\G)\cup \II(\G)\cup X)\\
&\,-8\#(\Fsf_\bdd(\G)\setminus S)+\sum_{F\in S}\delta(F)+6\# \Ed^{\ext}_S(\G).
\end{aligned}
\end{equation}
\end{theorem}


\begin{proof}
First, recalling \eqref{eq:rewriting_F_defects}, we use a double counting to write the sum of the angular defects among all vertices as the sum of the defects over all faces, and we split the sum according to $S$:
\begin{equation}\label{stardeco}
\F(\G)=\sum_{x\in X}\delta(x)=\sum_{F\in\Fsf(\G)}\delta(F)
=\sum_{F\in S}\delta(F)+\sum_{F\in \Fsf(\G)\setminus S}\delta(F).
\end{equation}
The first sum on the right-hand side term of \eqref{stardeco} is already one of the parts of the final energy decomposition, so we just need to expand the second one. 
Recalling \eqref{eq:delta_F} and \eqref{eq:delta_F_unbounded} we obtain
\begin{equation}\label{dallefacce}
\begin{aligned}
\sum_{F\in \Fsf(\G)\setminus S}\delta(F)=&\,\sum_{F\in \Fsf(\G)\setminus S}3\Per_\comb(F)+8\sum_{F\in \Fsf(\G)\setminus S}\#\Ccal^{\Int}(\partial F)\\
&\qquad -8\#(\Fsf_{\bdd}(\G)\setminus S).
\end{aligned}
\end{equation}
For the first summand in \eqref{dallefacce} we obtain
\[
\sum_{F\in \Fsf(\G)\setminus S}3\Per_\comb(F)=3\# \Ed_S^\partial(\G)+6\# \Ed^{\ext}_S(\G)=3\Per_\comb(A_S(\G))+6\# \Ed^{\ext}_S(\G),
\]
since the edges in $\Ed_S^\partial (\G)$ are exactly those that appear in the perimeter of $A_S(\G)$ (and are thus counted only once), while $\Ed^{\ext}_S(X)$ are those counted twice. For the second summand in \eqref{dallefacce}, Lemma \ref{lemma:double_counting} yields
\[
\sum_{F\in \Fsf(\G)\setminus S}\#\Ccal^{\Int}(F)=\#\Ccal(A_S(\G)\cup\II(\G)\cup X).
\]
Putting all the terms together we obtain the desired decomposition.
\end{proof}
\begin{remark}
In what follows we will apply the energy decomposition \eqref{formuladeco} to the graph ``generated''  by configuration whose energy $\E$ \eqref{introen} is finite (see formula \eqref{generated_edges}). By Lemmas \ref{crossquadr} and \ref{lemma:minimal_angle} below, such graphs are admissible, i.e, satisfy both properties \eqref{crossass} and \eqref{maxdeg}.
We stress that to get Theorem \ref{thm:decomposition_square} it is enough to assume that only property \eqref{crossass} is satisfied. On the other hand, property \eqref{maxdeg} guarantees that the angular defect of a face is nonnegative and such a property is crucial when proving crystallization. 
\end{remark}
\begin{remark}\label{damettere}
   We highlight that the energy decomposition \eqref{formuladeco} is purely combinatorial and is not related to the choice of a particular energy functional. It can be applied for instance to the pairwise interaction energy defined in \eqref{llm}. As an example of its ductility, we will see, in Subsection \ref{decorettri} below, how this formula can be adapted also to the case of a graph arising from the standard (isotropic) sticky disk.
\end{remark}
\subsection{Energy decomposition in the triangular case}\label{decorettri}
We report here also a version of the energy decomposition \eqref{formuladeco} that works for the ``triangular'' setting. We omit the proof because it is virtually identical to the one above. 

Let then $\G=(X,\Ed)$ be a planar graph with straight edges, with maximum degree 6. Given an angle $\alpha$, define its \textit{triangular defect} as
\begin{equation}\label{eq:triangular_defect}
\delta_\triangle(\alpha):=\frac{3}{\pi}\alpha -1.
\end{equation}
Accordingly, we let $\delta_\Delta(F)$ be the defect of the face $F$, computed replacing $\delta(\alpha)$ with $\delta_\Delta(\alpha)$.
Define as above $\Fsf(\G)$ as the set of faces of $\G$, and given $S\subseteq \Fsf(\G)$ let $A_S(\G)$ and $\Ed_S^{\ext}(\G)$ and $\Ccal (A_S(\G)\cup\II(\G)\cup X)$ be defined as in Subsection \ref{subsec:energy_decomposition}.
Defining $\F(\G):=\sum_{x\in X}(6-\deg_\G(x))$, we have the following energy decomposition.

\begin{theorem}[Energy decomposition for 6 neighbors]\label{thm:decomposition_triangular}
    For every planar graph $\G$ with maximum degree 6, and any choice of a subset $S\subseteq\Fsf(\G)$, the following equality holds:
\begin{equation*}
\begin{aligned}
    \F(\G)=&\, 2\Per_\comb(A_S(\G))+6\#\Ccal\big(A_S(\G)\cup\II(\G)\cup X\big)\\
    &\,-6\#(\Fsf(\G)\setminus S)+\sum_{F\in S}\delta_\triangle(F) +4\# \Ed_S^{\mathrm{ext}}(\G).
\end{aligned}
\end{equation*}
\end{theorem}

\begin{remark}[Relation with the energy decomposition in {\cite{DeLucaFriesecke}}]\label{rmk:decomposition}
We observe that the energy decomposition given in \cite{DeLucaFriesecke} corresponds to the choice $S=\Fsf(\G)$ in Theorem \ref{thm:decomposition_triangular}, whenever the faces are all simply connected: in this case we obtain
\begin{align*}
\F(\G)&=2\Per_\comb(A(\G))+6\#\Ccal(A(\G)\cup\II(\G)\cup X)+\sum_{F\in\Fsf(\G)}\delta_\triangle(F)+4\# \Ed^{\ext}_S(X)\\
&=2\Per_\comb(\G)+6\#\Ccal(A(\G)\cup\II(\G)\cup X)+2\mu(\G).
\end{align*}
Here, following the notation in \cite{DeLucaFriesecke}, $\mu(\G)$ stands for the number of additional edges needed to triangulate all the bounded faces of the graph (indeed, notice that $\delta(F)=2k_F-6$ is twice the number of sides needed to triangulate $F$), and $\Per_\comb(\G)$ is the combinatorial perimeter that counts exterior wires twice. This corresponds exactly to the energy decomposition obtained in \cite[Theorem 3.1]{DeLucaFriesecke} (notice also that in their notation, under the assumption that all the faces are simply connected, $\chi(X)$ coincides with the number of connected components of the graph $\G$, namely, with $\#\Ccal(A(\G)\cup\II(\G)\cup X)$; moreover $\F(\G)=6\sharp X+2\E_{\mathrm{HR}}(X)$).

Clearly, the decomposition of the energy in Theorem \ref{thm:decomposition_triangular} must be equivalent to that obtained in \cite[Theorem 3.1]{DeLucaFriesecke}, also without assuming that all the faces are simply connected, since both decompositions are a rewriting of $\F(\G)$; however, in such a general case the analogy is less easy to see, since the notion of face in \cite{DeLucaFriesecke}  does not coincide with that given here, and this choice clearly affects $A(\G)$, and, in turn,
both the perimeter term and the number of connected components of $A(\G)$. In other words, the sum between the perimeter, the cardinality of connected components and the defects of the faces is the same in both the decomposition, although there is no correspondence between each of the single terms of the sums.
\end{remark}

\section{Crystallization}\label{sec:crystallization}
We now specialize the discussion to the square sticky disk. Our aim in this section is to prove that minimizers of $\E$ are crystallized for $N\ge 6$. We will follow very closely some ideas introduced by Brass \cite{Brass}. He showed that every finite energy configuration $X_N$ can be deformed into a crystallized configuration $\widetilde X_N$ without increasing the energy. In particular, this proves that there \textit{exist} crystallized minima, but does not immediately show that \textit{all} minima are necessarily crystallized. We analyze Brass's proof to show that, if the starting configuration is a minimum, and if it is not already crystallized, then the deformation strictly decreases the energy (thus contradicting the minimality assumption).

We define the space of admissible configurations by
\begin{equation*}
\A:=\{X\subset\R^2\,:\,\# X\in\N\,,\quad \|x-y\|_{\infty}\ge 1\textrm{ for every } x,y\in X\textrm{ with }x\neq y\}\,,
\end{equation*}
and we observe that 
$\A\equiv\{X\subset\R^2\,:\,\E(X)<\infty\}$\,.
Moreover, for every $N\in\N$\,, we define
\begin{equation*}
\A_N:=\{X\in\A\,:\,\# X=N\}\,,
\end{equation*}
so that $\A\equiv\bigcup_{N\in\N}\A_N$\,.
Analogously, we set
\begin{equation*}
\A^{\Z^2}:=\{X\in\A\,:\,X\subset\Z^2\}\qquad \textrm{and}\qquad \A^{\Z^2}_N:=\A^{\Z^2}\cap\A_N\,.
\end{equation*}
Our main result in this section is the following.
\begin{theorem}\label{mainsquare}
Let $N\in\N$ and let $X_N$ be a minimizer of $\E$ in $\A_N$. Then $\G(X_N)$ is connected. Moreover, if $N\ge 3$, then $\G(X_N)$ has no wire edges. Furthermore, if $N\ge 6$,
then $X_N\in \A_N^{\Z^2}$ (up to a translation), $A(X_N)$ has simple and closed polygonal boundary, and $\delta(F)\in\{0,1\}$ for every bounded face $F\in\Fsf^{\bdd}(X_N)$.   
\end{theorem}
\begin{remark}
Notice that the assumption $N\ge 6$ is necessary as the result is not satisfied for $N=2,3,5$.
Indeed, the minimizers of $\E$ for $N=2$ are the pairs of points $\{x,y\}$ with $\|x-y\|_{\infty}=1$ and for $N=3$ are the triple of points $\{x,y,z\}$ forming an 
equilateral (with respect to $\|\cdot\|_\infty$) triangle with sidelength (with respect to $\|\cdot\|_\infty$) equal to $1$. Therefore, for $N=2$ and $N=3$ minimizers do not necessarily lie on a copy of $\Z^2$ (see Figure \ref{fig:stima} below for an example for $N=3$).

One can easily check that the unique (up to a translation) minimizer of the energy $\E$ in the case of $4$ particles is the configuration $\overline X_4=\{(0,0),(1,0),(1,1),(0,1)\}$, so in this case the minimizer is in $\A_4^{\Z^2}$. 
Finally for $N=5$, one has two classes of minimizers. The first one, belonging to in $\A_5^{\Z^2}$, is given --up to a translation-- by the configuration $\{(0,0),(1,0),(0,1),(-1,0), (0,-1)\}$; the second class is given by the configurations $\overline X_4\cup\{\bar x\}$, where the point $\bar x\notin[0,1]^2$ is positioned in such a way that it forms an equilateral (with respect to $\|\cdot\|_\infty$) triangle with sidelength (with respect to $\|\cdot\|_\infty$) equal to $1$ with one of the horizontal or vertical bonds of $\overline X_4$
(see Figure \ref{nonvero}).
\end{remark}
\begin{figure}[htbp]
    \centering
{\def\svgwidth{210pt}
\begingroup%
  \makeatletter%
  \providecommand\color[2][]{%
    \errmessage{(Inkscape) Color is used for the text in Inkscape, but the package 'color.sty' is not loaded}%
    \renewcommand\color[2][]{}%
  }%
  \providecommand\transparent[1]{%
    \errmessage{(Inkscape) Transparency is used (non-zero) for the text in Inkscape, but the package 'transparent.sty' is not loaded}%
    \renewcommand\transparent[1]{}%
  }%
  \providecommand\rotatebox[2]{#2}%
  \newcommand*\fsize{\dimexpr\f@size pt\relax}%
  \newcommand*\lineheight[1]{\fontsize{\fsize}{#1\fsize}\selectfont}%
  \ifx\svgwidth\undefined%
    \setlength{\unitlength}{850.39370079bp}%
    \ifx\svgscale\undefined%
      \relax%
    \else%
      \setlength{\unitlength}{\unitlength * \real{\svgscale}}%
    \fi%
  \else%
    \setlength{\unitlength}{\svgwidth}%
  \fi%
  \global\let\svgwidth\undefined%
  \global\let\svgscale\undefined%
  \makeatother%
  \begin{picture}(1,0.43333333)%
    \lineheight{1}%
    \setlength\tabcolsep{0pt}%
    \put(0,0){\includegraphics[width=\unitlength,page=1]{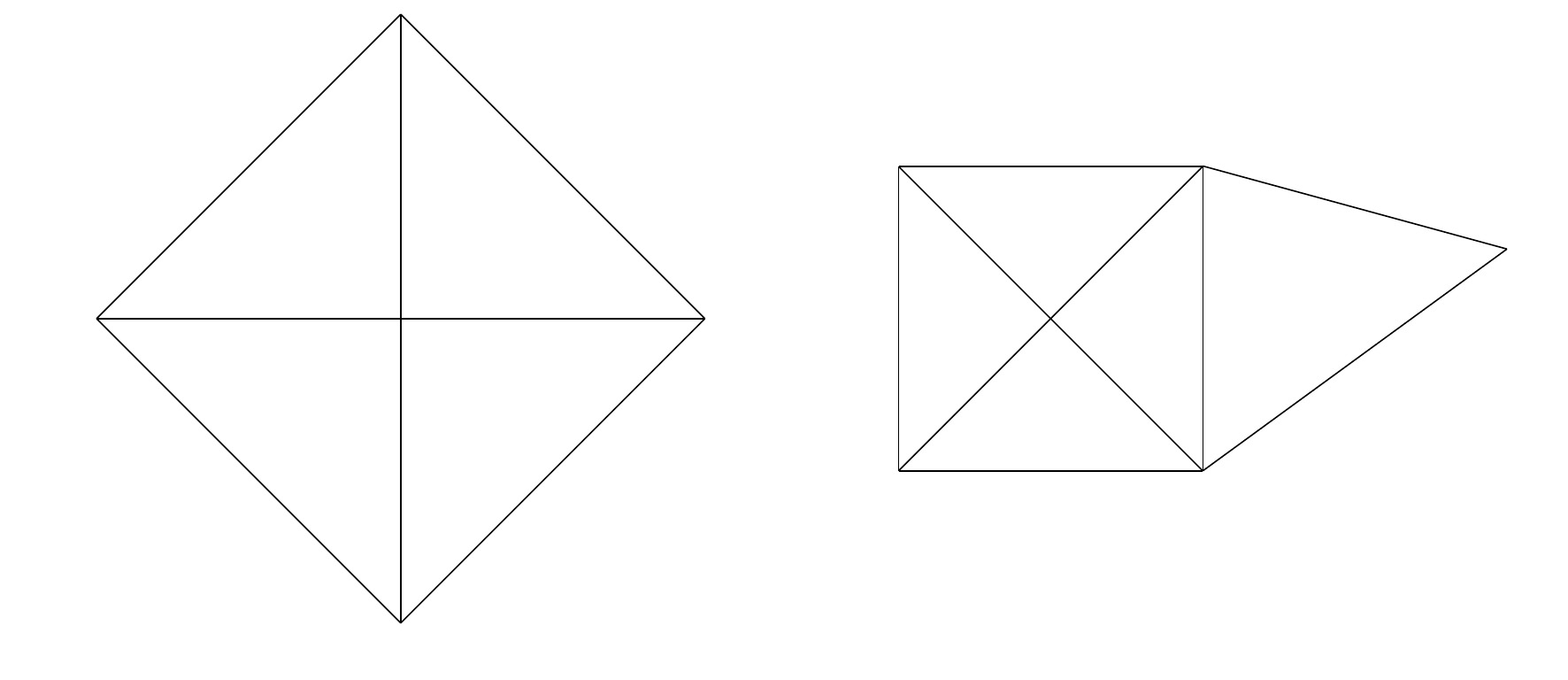}}%
    \put(0.88762992,0.30698237){\color[rgb]{0,0,0}\makebox(0,0)[lt]{\lineheight{1.25}\smash{\begin{tabular}[t]{l}$\bar x$\end{tabular}}}}%
    \put(0,0){\includegraphics[width=\unitlength,page=2]{nonvero.pdf}}%
  \end{picture}%
\endgroup%
}
    \caption{The two types of minimizers of the energy for $N=5$. On the left: the configuration $\{(0,0),(1,0),(0,1),(-1,0), (0,-1)\}$. On the right: the configuration $X_4\cup\{\bar x\}$, where $\bar x\in I:=\{(2,t)\,:\,t\in [0,1]\}$; the segment $I$ is represented by the dashed line.}
    \label{nonvero}
\end{figure}
In order to prove Theorem \ref{mainsquare} above we need to introduce some notation and to prove some auxiliary results.
For every $X\in\A$ we consider
the graph $\G(X)=(X,\Ed(X))$ generated by $X$ taking
\begin{equation}\label{generated_edges}
\Ed(X):=\{\{x,y\}\,:\,x,y\in X, \|x-y\|_\infty=1\}.
\end{equation}
In what follows, the dependence on $\G$ of some objects (such as, for instance, $\Fsf$) might be replaced, with a little abuse of notation, by the dependence on $X$.
We start with a simple but crucial observation: Minimizing $\E$ in $\A_N$ is equivalent to minimizing $2\E+8N$. By construction, for every $X\in\A_N$,
$$
2\E(X)+8N=\F_N(X),
$$
where $\F_N$ is the excess energy defined in \eqref{excess}.
Therefore, we will use all the objects and results introduced in Section \ref{sec:energy_decomposition} to prove our main theorem.

To this end, we prove Lemma \ref{crossquadr} and Lemma \ref{lemma:minimal_angle} below, showing that for every $X\in\A$, the graph $\G(X)=(X,\Ed(X))$
satisfies property \eqref{crossass} and \eqref{maxdeg}, so that $\G(X)$ is an admissible graph for every $X\in\A$.
\begin{lemma}[Crossing edges]\label{crossquadr}
Let $X\in\A$ and let $x_1,x_2,x_3,x_4\in X$ be such that $\{x_1,x_3\}, \{x_2,x_4\}\in\Ed(X)$ and $\#((x_1,x_3)\cap (x_2,x_4))=1$\,. Then the points $x_j$ are the vertices of a unit square with sides parallel to the cartesian axes.
\end{lemma}
\begin{proof}
Setting $|a|:=\sqrt{|a_1|^2+|a_2|^2}$ for any $a\in\R^2$\,, we start by noticing 
 that 
 \begin{equation}\label{norme}
 \|a\|_{\infty}\le |a|\le \sqrt 2\|a\|_{\infty}\qquad\qquad\textrm{for every }a\in\R^2,
 \end{equation}
 where the first inequality is an equality if and only if either $a=\lambda e_1$ or $a=\lambda e_2$ for some $\lambda\in\R$\,.
 
 For any $j,k=1,\ldots,4$ with $k\neq j$\,, we set $l_{j,k}:=|x_j-x_k|$\,. 
 By assumption, the quadrilateral $Q$ having vertices at the points $x_1,\,x_2,\, x_3,\,x_4$ is convex.
 Moreover, we denote by $\alpha_{i,j,k}$ the convex angle with vertex in $x_j$ that is spanned by the segments $(x_i,x_j)$ and $(x_j,x_k)$\,.
 Since $X\in\A$\,, 
 by the Cosine Theorem and by \eqref{norme}, we have
 \begin{equation}\label{unica}
 \begin{aligned}
2=&\,2\|x_1-x_3\|_\infty^2\ge l^2_{1,3}=\, l_{1,2}^2+l_{2,3}^2-2l_{1,2}l_{2,3}\cos \alpha_{1,2,3}\\
\ge&\,\|x_1-x_2\|_{\infty}^2+\|x_2-x_3\|_{\infty}^2-2l_{1,2}l_{2,3}\cos \alpha_{1,2,3}\\
\ge&\, 2-2l_{1,2}l_{2,3}\cos \alpha_{1,2,3}\,,
\end{aligned}
\end{equation}
whence we deduce that $0<\alpha_{1,2,3}\le \frac\pi 2$\,. 
Analogously, one can show that $0<\alpha_{2,3, 4},\,\alpha_{3,4,1},\, \alpha_{4,1,2}\le \frac\pi 2$\,, whence we deduce that the internal angles of the quadrilateral $Q$ are all equal to $\frac\pi 2$\,. It follows that $Q$ is a rectangle and that all the inequalities in \eqref{unica} are actually equalities. In particular, $l_{1,2}=\|x_1-x_2\|_{\infty}=1=\|x_2-x_3\|_{\infty} =l_{2,3}$ and $l_{1,3}=\sqrt 2$\,; therefore, $Q$ is a unit square with sides parallel to the cartesian axes.
\end{proof}

\begin{lemma}[Minimal angle]\label{lemma:minimal_angle}
    Let $X\in\A$. The following facts hold true.
    \begin{enumerate}
        \item A triangular equilateral face (with respect to $\|\cdot\|_\infty$) necessarily has either a horizontal or a vertical side (or both).
        \item  Let $x', x,x''\in X$ be such that $\{ x,x'\},\{x,x''\}\in\Ed(X)$. Then $\widehat{x' xx''}\ge\tfrac{\pi}{4}$, and equality holds if and only if  $x', x,x''$ are the vertices of an isosceles right triangle, with horizontal and vertical catheti of length 1.
        \item Given $x\in X$, we have that $\F(x)\ge 0$.
        Moreover, if $\F(x)=0$ then $x$ is crystallized, namely it is surrounded by four crystallized square faces belonging to $\Fsf_\boxtimes(X)$. In particular, $x$ does not belong to any planar face $F\in \Fsf(X)\setminus \Fsf_\boxtimes(X)$.
    \end{enumerate}
\end{lemma}
\begin{proof}
    (1) Up to symmetries, we can assume that two vertices of the face are the origin and the point $x=(1,x_2)$, with $0\le x_2\le 1$. We can assume that the first inequality is strict otherwise we are already done. It is easy to see that then the third point $y$ must be either $(1,x_2-1)$, or $(0,1)$. In the former case the side $[x,y]$ is vertical, in the latter the side $[0,y]$ is vertical, and in either case the conclusion follows.
    
    (2) Up to symmetries, we can assume that $x=0$ and $x'=(1,x_2)$ with $0\le x_2\le 1$. Then the minimal angle between $[x',x]$ and $[x,x'']$ appears when $\|x''\|_\infty=1$, namely when the points $0$, $x$ and $y$ form an equilateral triangle (with respect to the distance $\|\cdot\|_\infty$). Therefore we can assume that $x'$ and $x''$ are sides of an equilateral triangle, which - by item (1) - necessarily has at least one horizontal or vertical side. Let us assume without loss of generality that $x'=(1,0)$. Then, either $x''=(t,1)$ or $x''=(t,-1)$ with $0\le t\le 1$. Let us assume the former case by symmetry. Then the two angles at the base are given by $\alpha_1=\arctan \frac{1}{t}$ and $\alpha_2=\arctan\frac{1}{1-t}$. It is easily seen that $\alpha_1\ge\tfrac{\pi}{4}$, with equality if and only if $t=1$, and viceversa $\alpha_2\ge\tfrac{\pi}{4}$, with equality if and only if $t=0$. In either cases the conclusion follows. The third angle measures instead
    \[
    \pi-\arctan\frac{1}{t}-\arctan\frac{1}{1-t}.
    \]
    It is easy to see that this is a strictly concave function on $[0,1]$, and thus can reach its minimum only at the extremes, both of which correspond to an isosceles right triangle.

    (3) By item (2), we have that $\deg(x)\le 8$ for every $x\in X$, so that $\F(x)\ge 0$. 
   Now $\F(x)=0$ if and only if $\deg(x)=8$, which, again by item (2), implies that every angle between two edges at $x$ must be equal to $\tfrac{\pi}{4}$. It follows, again by item (2), that around $x$ eight isosceles right triangles appear, and thus $x$ is crystallized.
\end{proof}
Proposition \ref{basicprop} below proves the first part of the statement of Theorem \ref{mainsquare}.
\begin{proposition}\label{basicprop}
    Let $N\in\N$ and let $X_N\in\mathrm{argmin}_{X\in\A_N}\F_N(X)$.
    Then $\G(X_N)$ is connected. Moreover, if 
$N\ge 3$ then $\G(X_N)$, has no wire edges.
\end{proposition}
\begin{figure}[htbp]
    \centering
    \def\svgwidth{0.5\columnwidth}
\begingroup%
  \makeatletter%
  \providecommand\color[2][]{%
    \errmessage{(Inkscape) Color is used for the text in Inkscape, but the package 'color.sty' is not loaded}%
    \renewcommand\color[2][]{}%
  }%
  \providecommand\transparent[1]{%
    \errmessage{(Inkscape) Transparency is used (non-zero) for the text in Inkscape, but the package 'transparent.sty' is not loaded}%
    \renewcommand\transparent[1]{}%
  }%
  \providecommand\rotatebox[2]{#2}%
  \newcommand*\fsize{\dimexpr\f@size pt\relax}%
  \newcommand*\lineheight[1]{\fontsize{\fsize}{#1\fsize}\selectfont}%
  \ifx\svgwidth\undefined%
    \setlength{\unitlength}{308.70818907bp}%
    \ifx\svgscale\undefined%
      \relax%
    \else%
      \setlength{\unitlength}{\unitlength * \real{\svgscale}}%
    \fi%
  \else%
    \setlength{\unitlength}{\svgwidth}%
  \fi%
  \global\let\svgwidth\undefined%
  \global\let\svgscale\undefined%
  \makeatother%
  \begin{picture}(1,0.70050764)%
    \lineheight{1}%
    \setlength\tabcolsep{0pt}%
    \put(0,0){\includegraphics[width=\unitlength,page=1]{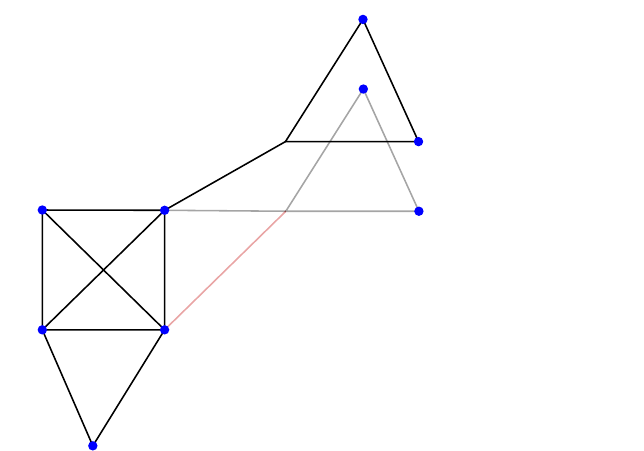}}%
    \put(0.23631269,0.40204852){\makebox(0,0)[lt]{\lineheight{1.25}\smash{\begin{tabular}[t]{l}$x$\end{tabular}}}}%
    \put(0.41047719,0.51069686){\makebox(0,0)[lt]{\lineheight{1.25}\smash{\begin{tabular}[t]{l}$x'$\end{tabular}}}}%
    \put(-0.0043022,0.17551141){\makebox(0,0)[lt]{\lineheight{1.25}\smash{\begin{tabular}[t]{l}$\mathsf{C}$\end{tabular}}}}%
    \put(0,0){\includegraphics[width=\unitlength,page=2]{square_graph_no_wires.pdf}}%
    \put(0.623836,0.66112987){\makebox(0,0)[lt]{\lineheight{1.25}\smash{\begin{tabular}[t]{l}$\mathsf{C}'$\end{tabular}}}}%
  \end{picture}%
\endgroup%

    \caption{Reference for the proof of Proposition \ref{basicprop}. The wire edge is $\{x,x'\}$. By ``sliding'' the whole component $\mathsf{C}'$ along $S(x)$ (dotted) we can create a new bond (shown in red) thus proving that the original configuration is not minimal.}
    \label{fig:no_wires}
\end{figure}
\begin{proof}
Assume by contradiction that $\G(X_N)$ has (at least) two connected components $\mathsf{C}$ and $\mathsf{C}'$\,. Then, there exists a translation $\mathcal T$ such that $\mathcal T(\mathsf{C}')$ is connected by (at least) one edge to $\mathsf{C}$\,, thus contradicting the minimality of $X_N$\,.
Therefore, all minimizers are connected.

Assume now that $N\ge 3$\,.  
Assume by contradiction that there exists a wire edge $\{x,x'\}\in \Ed(X_N)$\,. 
Without loss of generality we can assume that
$\deg(x)\ge 2$ and $\deg(x')\ge 2$\,.
Furthermore let $\mathsf{C}$ and $\mathsf{C}'$ be the two connected components of the graph $\G(X_N)$ obtained by removing from $\G(X_N)$ the wire edge $\{x,x'\}$ but keeping the vertices $x$ and $x'$\,. 
We can assume without loss of generality that $x\in\mathsf{C}$ and $x'\in\mathsf{C}'$\,. Therefore, setting $S(x):=\{z\in\R^2\,:\,\|z-x\|_{\infty}=1\}$ and letting $x'$\,, and consequently $\mathsf{C}'$\,, move along $S(x)$ we can create a new edge connecting $\mathsf{C}'$ to another connected component of $\G(X_N)$\,, thus contradicting the minimality of $X_N$\,.
\end{proof}
The next statement is a rewriting of Brass's results \cite{Brass} in our notation.
\begin{proposition}\label{isomorphism}
Let $X\in\A$. Then there exists
a map $\TT=\TT_{X}:X\to\R^2$ such that $\TT X\in \A^{\Z^2}$,
\begin{equation}\label{dibrass0}
    \|\TT x-\TT y\|_\infty\ge 1\qquad\textrm{for every } x,y\in X\textrm{ with }\|x-y\|_\infty\ge 1
\end{equation}
and
\begin{equation}\label{dibrass}
    \|\TT x-\TT y\|_\infty=1\qquad\textrm{for every } x,y\in X\textrm{ with }\|x-y\|_\infty=1.
\end{equation}
Moreover, if $X_N\in\mathrm{argmin}_{X\in\A_N}\F_N(X)$, then also $\TT X_N\in\mathrm{argmin}_{X\in\A_N}\F_N(X)$ and
$\TT$ is an isomorphism between $\G(X_N)$ and $\G(\TT X_N)$, that is, for every $x,y\in X_N$
\begin{equation}\label{isomor}
\|x-y\|_{\infty}=1\quad\text{ if and only if }\quad \|\TT x-\TT y\|_\infty=1.
\end{equation}
\end{proposition}
\begin{proof}
The first statement follows by the argument in \cite[page 209]{Brass}.
Now, if $X_N$ minimizes $\F_N$ in $\A_N$, in view of \eqref{dibrass0} and \eqref{dibrass}, we have that 
\begin{equation}\label{minim}
  \F_N(\TT X_N)\le \F_N(X_N) \le  \F_N(\TT X_N),
\end{equation}
and hence also $\TT X_N$ is a minimizer of $\F_N$ in $\A_N$. Finally, one implication in \eqref{isomor} is exactly \eqref{dibrass} whereas, in order to prove the opposite implication, we adopt a simple contradiction argument. Indeed, if there exist $ x',y'\in X$ such that $\|\TT x'-\TT y'\|_{\infty}=1$ and $\|x'-y'\|_{\infty}>1$, then, by \eqref{minim}, there should exist $ x'',y''\in X$ such that $\|\TT x''-\TT y''\|_{\infty}>1$ and $\|x''-y''\|_{\infty}=1$ thus contradicting \eqref{dibrass}.
\end{proof}
With Proposition \ref{isomorphism} in hand, we proceed as follows: we first prove (Proposition \ref{pieno}) the analogue of Theorem \ref{mainsquare} in $\A^{\Z^2}$; then, we prove some qualitative rigidity properties on the minimizers of $\F_N$ in $\A_N^{\Z^2}$, with $N\ge 6$, (from Lemma \ref{angoliconcavi} to Lemma \ref{ancorapiurig}) which show that, given a minimizer $X_N$ of $\F_N$ in $\A_N^{\Z^2}$, then $\F_N(\TT X_N)=\F_N(X_N)$ if and only if $X\in \A_N^{\Z^2}$. This argument is rigorously formalized at the end of this section.
The next lemma shows the monotonicity of the minimal value of the energy $\F_N$ with respect to the number $N$ of particles.
\begin{lemma}
\label{lemma:monotonicity}
    The function
    \[
    N\mapsto \min_{X\in \A_N} \F_N(X)
    \]
    is non-decreasing in $N$. Equivalently, the maximum number of edges for finite energy configurations with $N$ points can increase by at most $4$ when we pass from $N-1$ to $N$.
\end{lemma}
\begin{proof}
    Let $X_N$ be any minimizer of $\F_N$ in $\A_N$. There exists a vertex $x\in X_N$ with $\deg(x)\le 4$. Indeed, it is sufficient to take $x$ as a vertex of the convex hull of $X_N$: by Lemma \ref{lemma:minimal_angle}(2) such vertex cannot have degree bigger than $4$. Then $X':=X_N\setminus \{x\}$ is a competitor in $\A_{N-1}$ and
    \[
    \#\Ed(X')=\#\Ed(X_N)-\deg(x)\ge \# \Ed(X_N)-4.
    \]
    This proves the last part of the statement. It follows that
    \begin{align*}
    \min_{X\in \A_N}\F_N(X) =\F_N(X_N)&=8N-2\# \Ed(X_N)\\
    &\ge 8N -2(\# \Ed(X')+4)\\
    &=8(N-1)-2\# \Ed(X')\\
    &=\F_{N-1}(X')\\
    &\ge \min_{X\in\A_{N-1}}\F_{N-1}(X).
    \end{align*}
    The proof is complete.
\end{proof}

As a consequence of Lemma \ref{lemma:monotonicity}, we prove that, for every $N\ge 3$, minimizers of $\F_N$ in $\A^{\Z^2}_N$ have only triangular and crystallized square faces.
\begin{proposition}\label{pieno}
    Let $N\in\N$ with $N\ge 3$, and let
    \[
    X_N\in \mathrm{argmin}_{X\in \A_N^{\mathbb{Z}^2}} \F_N(X).
    \]
    Then every bounded face of $X_N$ is either a triangle or a crystallized square. 
\end{proposition}
\begin{proof}
By Proposition \ref{basicprop}, we find that $\G(X_N)$ is connected and has no wire edges.

Assume by contradiction that there exists a face $F_0$ that is neither a triangle nor a crystallized square.
Since, for any given $X\in\A^{\Z^2}$, 
$\Per_\comb(F)\neq 5$ for every face $F\in\Fsf(X)$, we have either $\Per(F_0)\ge 6$ or $\Per(F_0)=4$.

Assume first that $\Per(F_0)=k$ with $k\ge 6$. 
Let 
$\{x^j(F_0)\}_{j=1,\ldots,J}$ denote the 
set of points in $\Z^2$  that do not belong $X_N$ but are contained in the (interior of the) face  $F_0$. Since $\Per(F_0)\ge 6$, we have that $J\ge 2$.
We consider the configuration $X':=X_N\cup\{x^j(F_0)\}_{j=1,\ldots,J}$.
We claim that 
\begin{equation}\label{contraddi}
\F_{N+J}(X')<\F_N(X_N),
\end{equation}
which contradicts the monotonicity proven in Lemma \ref{lemma:monotonicity}.
From the energy decomposition of Theorem \ref{thm:decomposition_square}, by choosing $S=\Fsf_\bdd(X_N)$ as the set of all bounded faces, we obtain
    \[
    \F_N(X_N)=3 \Per(A(X_N))+8
    +\sum_{F\in \Fsf(X_N)}\delta(F)\,,
    \]
    A similar decomposition holds for $X'$. Observe that all faces of $X'$ coincide with those of $X_N$, except possibly for those that intersect $F_0$. Moreover, the perimeter term is the same for $X_N$ and $X'$. We deduce that
    \[
    \F_N(X_N)-\F_{N+J}(X')\ge \delta(F_0)-\sum_{\substack{F\in \Fsf(X')\\F\subset F_0}}\delta(F).
    \]
    On the one hand, we have $\delta(F_0)=3k-8$. On the other hand, all faces of $X'$ contained in $F_0$ are either crystallized squares or triangles, because we added all interior points. But there can be at most one triangle for every edge in the boundary of $F_0$. Together with the fact that, for a triangle $T$, $\delta(T)=1$, and using that $k\ge 6$,   it follows that
    \[
    \delta(F_0)-\sum_{\substack{F\in \Fsf(X')\\F\subset F_0}}\delta(F)\ge 3k-8-k=2k-8\ge 4,
    \]
    i.e., \eqref{contraddi}. This excludes the existence of a face in $X_N$ with $\Per(F_0)\ge 6$.  
    

    Assume now that $\Per(F_0)=4$. The only possible case (since $F_0$ is not a crystallized square) is that $F_0$ is a diamond, i.e., (up to translations) the square $D$ with vertices in $(1,0),(0,1),(-1,0),(0,-1)$. We first observe that neither of the four points $(\pm 1,\pm 1)$ can belong to $X_N$, since, otherwise, by adding the origin we would add five edges, and thus contradict Lemma \ref{lemma:monotonicity}. Then we observe that $X_N$ cannot consist of diamonds only: if this were the case we could find a vertex of degree $2$ (just take a convex angle of the outer boundary polygon); suppose without loss of generality that this is the vertex $(1,0)$ of $D$. Then by removing this vertex and adding it back in the origin we would gain at least one edge, thus contradicting the minimality of $X_N$. By connectedness, we can thus assume that there is some triangle or some crystallized square that shares a vertex with a diamond face, which we can assume to be $F_0=D$, and we assume that the shared vertex is $(1,0)$. Then the point $(2,0)$ belongs to $X_N$, and either $(2,1)$ or $(2,-1)$ must belong to $X_N$ as well. Assume by symmetry that the first case happens. Then we can add the points $(0,0)$ and $(1,1)$ to the configuration, and the total number of edges increases by at least 9. This contradicts again Lemma \ref{lemma:monotonicity}.

    In conclusion, the only faces that can appear in $X_N$ are triangles and crystallized squares.
\end{proof}
Lemma \ref{angoliconcavi} below shows that a minimizing configuration cannot have concave angles equal to $\frac\pi 2$.
\begin{figure}[htbp]
    \centering
{\def\svgwidth{115pt}
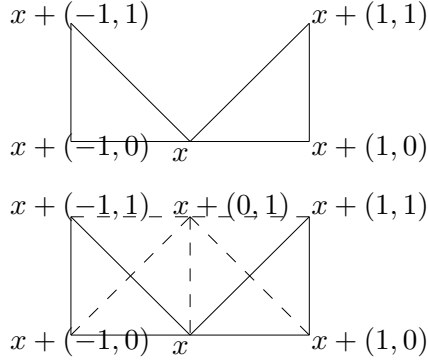}
    \caption{Proof of Lemma \ref{angoliconcavi}.}
    \label{ang_sbagl}
\end{figure}
\begin{lemma}\label{angoliconcavi}
Let $N\in\N$ with $N\ge 6$ and let $X_N$ be a minimizer of $\F_N$ in $\A_N^{\Z^2}$. Let $x\in\partial X_N$. 
The following facts hold true.
\begin{itemize}
\item
Assume that  $x+ (0,1)\notin X_N$ (resp., $x-(0,1)\notin X_N$). Then,  $x+(1,1)\in  X_N$ (resp. $x-(1,1)\in X_N$) if and only if $x+(-1,1)\notin X_N$ (resp. $x-(-1,1)\in X_N$). 
\item
Assume that   $x+ (1,0)\notin X_N$ (resp., $x-(1,0)\notin X_N$). Then,  $x+(1,1)\in X_N$ (resp. $x-(1,1)\in X_N$) if and only if $x+(1,-1)\notin X_N$ (resp. $x-(1,-1)\in X_N$). 
\end{itemize}
\end{lemma}
\begin{proof}
By symmetry, it is enough to prove only one implication in one of the case listed above.

Assume by contradiction that $x':=x+(1,1)$ and $x'':= x+(-1,1)$ are both in $X_N$. We set $\tilde x:= x+(0,1)$. 
By Proposition \ref{basicprop} and by Proposition \ref{pieno} we have that $x\pm (1,0)\in X_N$.
Let $y\in\partial X_N\setminus\{x\pm (1,0), x', x''\}$
with $\deg(y)\le 4$; by Proposition \ref{pieno}, we have that the configuration $\widetilde X_N:=\big(X_N\setminus\{y\}\big)\cup\{\tilde x\}$ belongs to $\A_N^{\Z^2}$ and is such that $\deg(\tilde x)\ge 5$.
It follows that $\F_N(\widetilde X_N)\le \F_N(X_N)-2$, thus contradicting the minimality of $X_N$.
This concludes the proof of the result.
\end{proof}
Lemma \ref{angoliconcavi} ensures the absence of ``bow points'' in minimizing configurations, which allows to prove the ``regularity'' of the boundary of minimal configurations.
\begin{proposition}[Minimizers in $\Z^2$ have simple and closed polygonal boundary]\label{simplyclosedbdryy}
Let $N\in\N$ and let $X_N$ be a minimizer of $\F_N$ in $\A_N^{\Z^2}$. Then $A(X_N)$ has simple boundary.
\end{proposition}
\begin{proof}
The minimizers of $\F_N$ in $\A_N^{\Z^2}$ with $N\le 5$ are explicit (see Figure \ref{minimiexp}).
\begin{figure}[htbp]
    \centering
{\def\svgwidth{210pt}
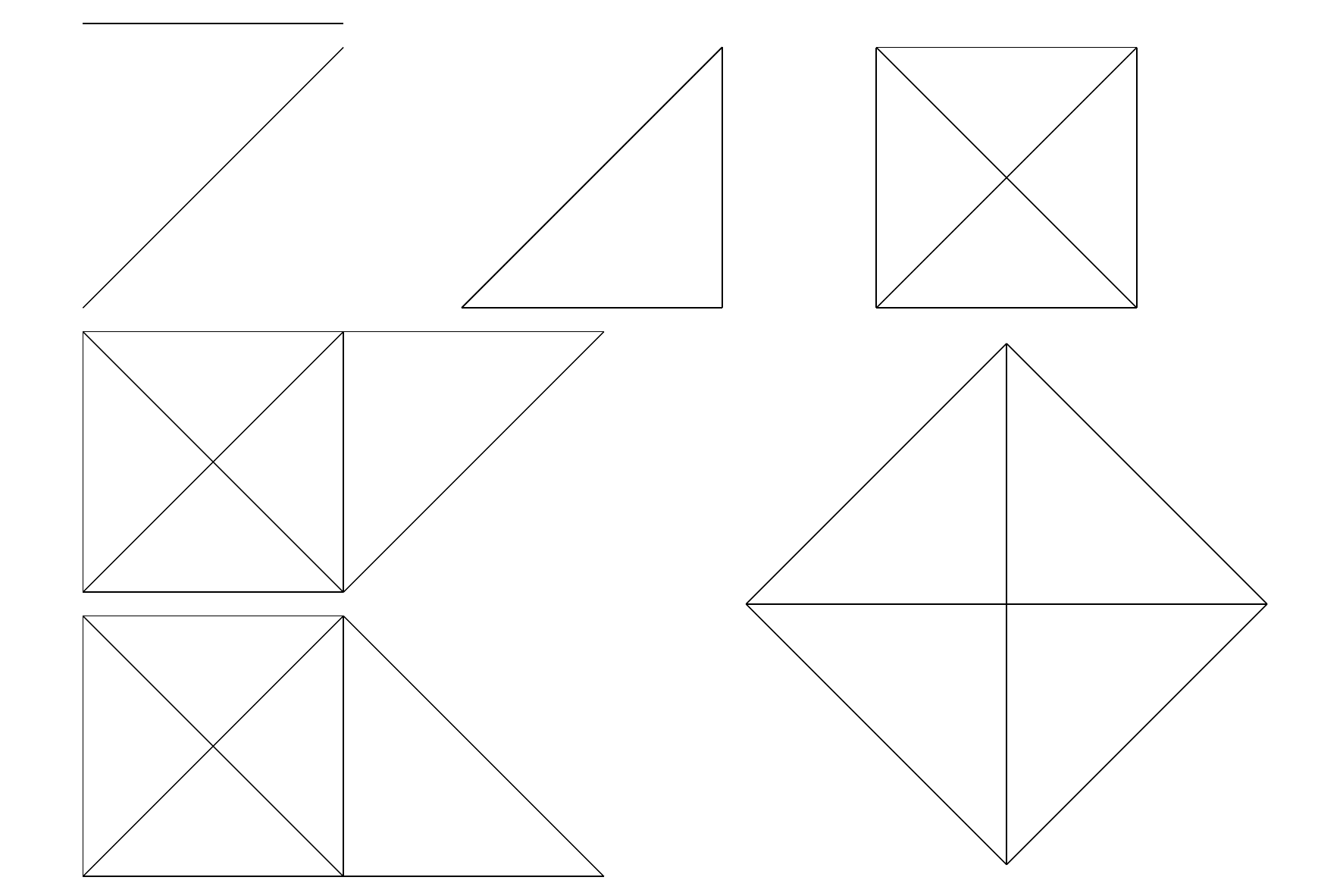}
    \caption{All the minimizers (up to a $\frac{\pi}{2}$ rotation) of the energy $\E$ in $\mathcal A_N^{\Z^2}$ for $N=2,3,4,5$.}
    \label{minimiexp}
\end{figure}
Let $N\ge 6$ and 
assume by contradiction that there exists $\bar x\in X_N$ such that $A(X_N)\setminus\bar x$ is disconnected.
Let $\Gamma',\Gamma''$ be the two connected components of $A(X_N)\setminus\{\bar x\}$ and set $X_N':=\Gamma'\cap X_N$ and $X_N'':=\Gamma''\cap X_N$. 

Since $X_N\in \A^{\Z^2}$, in view of Lemma \ref{angoliconcavi} we can assume, without loss of generality, that the points $\bar x\pm(1,0), \bar x\pm(1,1)$  are all in $\partial X_N$ whereas the points $\bar x\pm(-1,1),\bar x \pm(0, 1)\notin X_N$. We set $\overline X_N:=X_N'\cup\{\bar x+(0,-1)\}\cup (X_N''+(0,-1))$. By construction, $\overline X_N\in \A_N^{\Z^2}$ and $\F(\overline X_N)\le \F(X_N)$. If the previous inequality is strict we have obtained a contradiction with the minimality of $X_N$; if $\F_N(\overline X_N)= \F_N(X_N)$, we have that $\overline X_N$ is a minimizer of $\E$ in $\A_N^{\Z^2}$ that satisfies $\bar x+(-1,0),\bar x+(0,-1),\bar x+(1,0)\in \partial X_N$, thus contradicting Lemma \ref{angoliconcavi}.  
\end{proof}
Lemma \ref{gradomag3} and Lemma \ref{ancorapiurig} allow to prove that the triangular faces appearing in energy minimizers in $\Z^2$ are rigid; more precisely, either both the catheti lie on two squared crystallized faces, or if a cathetum is shared with a triangular face, then the union between the two adjacent triangles forms a triangle.
\begin{lemma}\label{gradomag3}
Let $N\in\N$ with $N\ge 6$ and let $X_N$ be a minimizer of $\F_N$ in $\A_{N}^{\Z^2}$. Then, $\deg(x)\ge 3$ for every $x\in X_N$.
\end{lemma}
\begin{proof}
Assume by contradiction that there exists $\bar x\in X_N$ with $\deg(x)= 2$ and let $X_N':=X_N\setminus\{\bar x\}$. Then, $\sharp X'_N\ge 5$ and $\F_{N-1}(X'_N)=\F_N(X_N)-4$.
If $\Fsf_\triangle(X'_N)=\emptyset$, then, since $\sharp X'_N\ge 5$, the polygon $A(X'_N)$ has a horizontal (resp., vertical) side with length $\ge 2$. Assume, without loss of generality, that the points $(-1,0),(0,0), (1,0)\in \partial X_N'$ and that the set $\{(t,s)\,:\,-1<t<1, -1<s<0\}$ is contained in $A(X'_N)$; then, the configuration $\bar X_N:=X'_N\cup\{(0,1)\}$ has energy $\F_N(\bar X_N)=\F_{N-1}(X_N')+2=\F_N(X_N)-2$, thus contradicting the minimality of $X_N$.

On the other hand, assume that there is a triangular face  $T=(x_1,x_2,x_3)\in \Fsf_\triangle(X'_N)$, and let $\hat x$ be such that the quadrilateral $(x_1,x_2,x_3,\hat x)$ is a unit square with vertices in $\Z^2$.
Then, taking $\widehat X_N:=X'_N\cup\{\hat x\}$ where $\hat x$ we have again that $\F_N(\widehat X_N)=\F_{N-1}(X_N')+2=\F_N(X_N)-2$, which still contradicts the minimality of $X_N$.

This concludes the proof.
\end{proof}
\begin{lemma}\label{ancorapiurig}
Let $N\in\N$ and let $X_N$ be a minimizer of $\F_N$ in $\A_{N}^{\Z^2}$. 

Assume that there are two triangular faces $T',T''$ sharing the edge $\{x,x+(1,0)\}$ for some $x\in X_N$. Then, either $T'=\{x,x+(1,0),x+(1,1)\}$ and $T''=\{x,x+(1,0), x+(1,-1)\}$ (and viceversa) or $T'=\{x,x+(1,0),x+(0,1)\}$ and $T''=\{x,x+(1,0), x+(0,-1)\}$ (and viceversa).

Analogously, assume that there are two triangular faces $T',T''$ sharing the edge $\{x,x+(0,1)\}$ for some $x\in X_N$. Then, either $T'=\{x,x+(-1,0),x+(0,1)\}$ and $T''=\{x,x+(1,0), x+(0,1)\}$ (and viceversa) or $T'=\{x,x+(-1,1),x+(0,1)\}$ and $T''=\{x,x+(1,1), x+(0,1)\}$ (and viceversa).
\end{lemma}
\begin{proof}
Assume by contradiction that $T'=\{x,x+(1,0),x+(1,1)\}$ and $T''=\{x,x+(1,0), x+(0,-1)\}$ and let $\G_N'$ and $\G_N''$ be the connected components of the graph $\bar G_N:=(X_N,\overline \Ed(X_N))$, where $\overline\Ed(X_N):=\Ed(X_N)\setminus\{\{x,x+(1,1)\},\{x,x+(1,0)\}, \{x+(-1,0),x+(1,0)\}\}$. 
Let $X_N'$ and $X_N''$ be the set of vertices of $\G_N'$ and $\G_N''$ respectively. 
We can assume without loss of generality that $x+(1,0), x+(1,1)\in X_N'$ and $x+(0,-1), x\in X_N''$.
Notice that, by assumption, $x+(1,-1)\notin X_N$.

We define $\widetilde X_N:=X_N''\cup (X_N'+(0,-1))$. Then, $\widetilde X_N\in\A_{N}^{\Z^2}$ and $\F_N(\widetilde X_N)\le \F_N(X_N)-2$ (see Figure \ref{dimolemmapiurig}), thus contradicting the minimality of $X_N$.
\end{proof}
\begin{figure}[htbp]
    \centering
{\def\svgwidth{210pt}
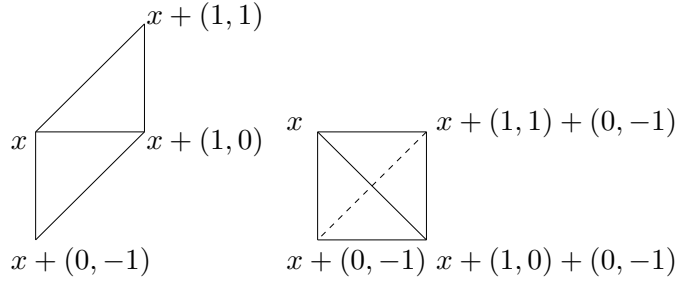}
    \caption{Proof of Lemma \ref{ancorapiurig}.}
    \label{dimolemmapiurig}
\end{figure}
\begin{remark}\label{rigidita}
By Lemma \ref{gradomag3} and Lemma \ref{ancorapiurig} we have that, given  a minimizer of $\F_N$ in $\A_N^{\Z^2}$ (with $N\ge 6$), its triangular faces may satisfy only one of the following conditions:  
either both its horizontal and vertical  edges belong to two distinct (crystallized) square faces or the triangular face shares one edge with another triangular face in such a way that the union of the two faces is still a triangle; i.e., up to $\tfrac \pi 2$ rotations, only the situations depicted in Figure \ref{faccetriz2} can occur.
\end{remark}

\begin{figure}[htbp]
    \centering
{\def\svgwidth{210pt}
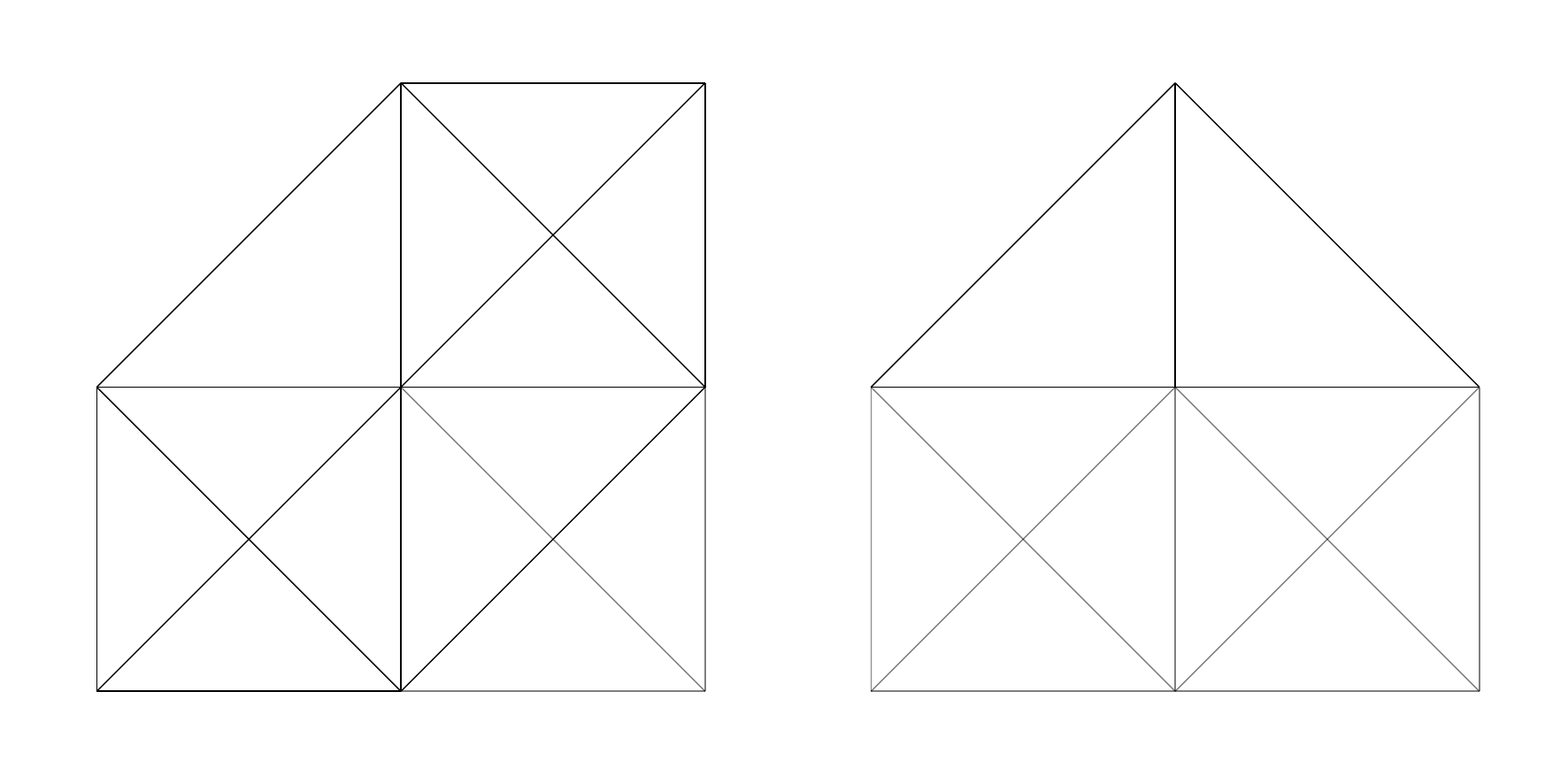}
    \caption{The triangular crystallized faces of minimizers of the energy $\F_N$ in $\A_N^{\Z^2}$.}
    \label{faccetriz2}
\end{figure}
Now we have all the ingredients to prove Theorem \ref{mainsquare}.
\begin{proof}[Proof of Theorem \ref{mainsquare}]
Let $N\in\N$ and let $X_N$ be a minimizer of $\F_N$ in $\A_N$. 
By Proposition \ref{basicprop} we have that $\G(X_N)$ is connected and, if $N\ge 3$, $\G(X_N)$ has no wire edges.

Assume now that $N\ge 6$. Let $\TT=\TT_{X_N}$ be the map provided by Proposition \ref{isomorphism}; we set $Z_N:=\TT X_N$. Then, by Proposition \ref{isomorphism}, $Z_N$ is a minimizer of $\F_N$ in $\A_{N}^{\Z^2}$ and, for every $x\in X_N$, we have $\F(x)=\F(\TT x)$. 
By Proposition \ref{pieno}, we have that the faces of $Z_N$ are either (crystallized) triangles or crystallized squares and, by Proposition \ref{simplyclosedbdryy}, $A(Z_N)$ has simple boundary.
Hence,
also $A(X_N)$ has and all the faces of $X_N$ are either triangles or crystallized squares.

Since $\TT$ maps crystallized square faces in crystallized square faces, the only thing to check is whether there is some  triangular face in $\Fa(X_N)$ whose vertices do not belong to (any translation of) $\Z^2$.
But this cannot happen since, in view of Lemma \ref{ancorapiurig} and Remark \ref{rigidita}, the triangular faces of $Z_N$ are rigid. 
This concludes the proof of Theorem \ref{mainsquare}.
\end{proof}
\section{\texorpdfstring{$\Gamma$}{Gamma}-convergence}\label{sec:gamma_convergence}
In this final section we prove the $\Gamma$-convergence result for the energy $\F_N(\cdot)=\E(\cdot)+4N$.
We will show that the $\Gamma$-limit of the functionals $N^{-1/2}\F_N$ as $N\to +\infty$ will be given by the anisotropic perimeter
\begin{equation}\label{eq:def_anisotropy}
\Per_\phi(E):=\int_{\partial^*E} \phi(\nu_E(x)) \ud\Hcal^1(x),
\end{equation}
where the map $\phi$ is defined in \eqref{finsler}.
To this end, we recall (see\eqref{eq:A(X)_definition} and  \eqref{defAS}) that for every $X\in\A$,  
\[
A_{\boxtimes}(X):=\bigcup_{F\in F_\boxtimes(X)}F=A_{\Fsf_{\boxtimes}(X)}(X).
\]
\begin{theorem}
\label{thm:gamma_convergence}
    The functionals $N^{-1/2}\F_N$ $\Gamma$-converge to $\Per_\phi$, namely:
    \begin{enumerate}
       \item 
        ($\Gamma$-liminf inequality) For every sequence of configurations $\{X_N\}_{N\in\N}\subset\A$, with $X_N\in\A_N$ (for every $N$), such that the sets $N^{-1/2}A_\boxtimes(X_N)$ converge locally in measure to some finite perimeter set $E\subset \R^2$, it holds
        \[
        \Per_\phi(E)\le\liminf_{N\to\infty} N^{-1/2}\F_N(X_N);
        \]
       \item 
       ($\Gamma$-limsup inequality) For every finite perimeter set $E$, there exists a 
        sequence $\{X_N\}_{N\in\N}\subset\A$ with $X_N\in\A_N$ (for every $N$), such that the sets
        $N^{-1/2}A_\boxtimes(X_N)$ converge locally in measure to $E$ and
       \[
        \Per_\phi(E)\ge \limsup_{N\to\infty}N^{-1/2}\F_N(X_N).
        \]
    \end{enumerate}
\end{theorem}

We remark that, by constraining the particles to the lattice $\Z^2$, this result was already known (see, e.g., \cite[Theorem~1.1]{DelNinPetrache}). Indeed, in that setting our potential becomes equivalent to considering an energy that activates both first and second neighbors. The strength of our result is that we do not require any constraint on the particles.

Before proving Theorem \ref{thm:gamma_convergence}, we state and prove the corresponding compactness result. 

\subsection{Compactness}
\begin{proposition}[Compactness]\label{prop:compactness}
    Let $\{X_N\}_{N\in\N}\subset\A$ be a sequence such that $X_N\in \A_N$ with 
    \begin{equation}\label{bounden}
    \F_N(X_N)\le CN^{1/2},
    \end{equation}
    for some constant $C$ (independent of $N$). Then:
    \begin{enumerate}
        \item $\# \Fsf_\boxtimes(X)\ge N- C N^{1/2}$.
        \item There exists a subsequence $\{N_k\}_{k\in\N}$ such that 
        \[
        N_k^{-1/2}A_\boxtimes(X_{N_k})\to E
        \]
    locally in measure for some finite perimeter set $E$.
    \end{enumerate}
\end{proposition}
\begin{proof} 
We prove the two properties in order.

(1)  For every $N\in\N$ we have the following chain of inequalities:
\[
CN^{1/2}\ge \F_N(X_N)=\sum_{x\in X_N} \F(x)=\sum_{\substack{x\in X_N\\ \deg(x)\le 7 }} \F(x)\ge \#\{x\in X_N:\, \deg(x)\le7\}.
\]
Therefore, all points of $X_N$, except for at most $CN^{1/2}$, have 8 neighbors. By Lemma \ref{lemma:minimal_angle}(3) it descends that all such points are crystallized. In particular, to each such vertex $x$, we can associate the crystallized square face whose lower left vertex is $x$. This association is injective; hence it follows that $\# \Fsf_\boxtimes(X)\ge N- C N^{1/2}$.

(2)  We aim to show that $\Per(A_\boxtimes(X_N))\le CN^{1/2}$ (for every $N\in\N$). 
Now observe that if a point $x\in X_N$ has degree 8, then the edges containing $x$ cannot be part of $\partial A_\boxtimes(X_N)$, because the configuration is crystallized around $x$ by Lemma \ref{lemma:minimal_angle}(3). Therefore, for every edge $e=\{x,x'\}$ satisfying $[x,x']\subseteq \partial A_\boxtimes(X_N)$ we must have that both $x$ and $x'$ have degree at most 7. Moreover, for the same reason, a vertex $x$ cannot be part of more than 7 distinct edges in $\partial A_\boxtimes(X_N)$ (in fact even less, but 7 is good enough for the proof). Hence by Point (1)
\begin{align*}
    \Per(A_N)&=\#\{e\in \Ed(X_N):\, [e]\in \partial A_\boxtimes (X_N)\}\\
    &\le 7 \#\{x\in X_N:\, \deg(x)\le 7\}\\
    &\le 7C N^{1/2}.
\end{align*}
It follows that the sets $N^{-1/2} A_\boxtimes(X_N)$ have uniformly bounded perimeter. By the compactness of finite perimeter sets (see \cite[Theorem~3.39]{AFP}) we deduce that, up to a subsequence, they converge locally in measure to some finite perimeter set $E$. This finishes the proof.
\end{proof}

\subsection{Proof of Theorem \ref{thm:gamma_convergence}}\label{sec:gamma_liminf}
This subsection is devoted to the proof of Theorem \ref{thm:gamma_convergence}.
We start by proving the lower bound inequality Theorem \ref{thm:gamma_convergence}(i). 


The idea is to construct a link between $\F_N(X_N)$ and  the $\phi$-anisotropic perimeter of $A_S(X_N)$, for a suitable choice of $S\subset \Fsf(X)$, in order to deduce our result from the lower semicontinuity of $\Per_\phi$. We start with the following observation: From Theorem \ref{thm:decomposition_square}, by neglecting positive terms, we immediately obtain that
\[
\F_N(X_N)\ge 3\Per_\comb(A_S(X_N))-8\#(\Fsf(X_N)\setminus S_N).
\]
By suitably choosing $S_N$ to be the family of ``small perimeter faces'', we will ensure that the negative term (when rescaled by $N^{1/2}$) goes to zero. This almost gives the desired lower bound inequality, but with $3\Per_\comb(A_{S_N}(X_N))$ instead of $\Per_\phi(A_{S_N}(X_N))$. Observe that the anisotropic $\phi$-length of every unit segment in the norm $\|\cdot\|_\infty$ always lies in the interval $[3,4]$, reaching value 3 only for horizontal and vertical segments and the value $4$ only for the ``diagonal'' ones. The inequality is thus not optimal, as we are losing something on every diagonal edge. The key result to fix this is the following lemma, which proves that we can replace the combinatorial perimeter with the anisotropic one by exploiting the defect of the faces.
In the following, for any given segment $I$ we denote by $\ell_\phi(I)$ its $\phi$-length, that is its Euclidean length multiplied by $\phi(\tfrac{e^\perp}{|e^\perp|})$.
\begin{lemma}\label{lemma:delta_versus_P}
Let $X\in\A$ and let $F\in\Fsf(X)$. Let 
    $\mathsf{E}:=\{e_1,\ldots,e_M\}\subset\Ed(F)$, with $1\le M\le \sharp\Ed(F)$, be any collection of distinct edges of $\partial F$. 
    Then it holds  
    \begin{equation}\label{claimdelta}
    3M +\delta(F)\ge \sum_{m=1}^M \ell_\phi([e_m]).
    \end{equation}
\end{lemma}
\begin{proof}
    Observe first that, if $F\in\Fsf_\boxtimes(\G)$, then $\delta(F)=0$ and $\ell_\phi([e])=3$ for every boundary edge $e\in\Ed^\partial(F)$. Thus, the inequality holds as equality. 
    
    First suppose that $\Per_\comb(F)\ge 4$. By construction, this implies that $\sharp \Ed^{\partial}(F)\ge 4$. Then, by \eqref{eq:delta_F}, we have $\delta(F)\ge 3\Per_{\comb}(F)-8$. Moreover, by the very definition of $\phi$, we have $\ell_\phi([e_m])\le 4 $ for every $m=1,\ldots,M$. It follows that
    \[
    3M+\delta(F)\ge 3M+3\Per_{\comb}(F)-8\ge 3M+\Per_\comb(F)\ge 4M \ge \sum_{m=1}^M \ell_\phi([e_m]),
    \]
    and the conclusion is reached.
    
    We are left with the case $\Per_\comb(F)=3$, that is when $F$ is a triangle. In this case $\delta(F)=1$. There are three subcases:
    \begin{itemize}
        \item If $M=1$, then trivially $3+\delta(F)=4\ge \ell_\phi([e_1])$, which proves \eqref{claimdelta}.
        \item If $M=3$, then we are going to show that $\Per_\phi(F)=10$, 
        from which \eqref{claimdelta} follows with an equality. 
        By Lemma \ref{lemma:minimal_angle}(1) one of the sides of $F$ must be orthogonal. Up to symmetries we can assume that the vertices of $F$ are the points $(0,0)$, $(1,0)$ and $(t,1)$ for some $t\in [0,\tfrac12]$. Set $e_1:=\{(0,0), (1,0)\}$, $e_2:=\{(0,0), (t,1)\}$, $e_3:=\{(t,1), (1,1)\}$ and
        call $\alpha$ the angle between $[e_1]$ and $[e_2]$ and $\beta$ the angle between $[e_1]$ and $[e_3]$.
        \begin{figure}[htbp]
    \centering
{\def\svgwidth{120pt}
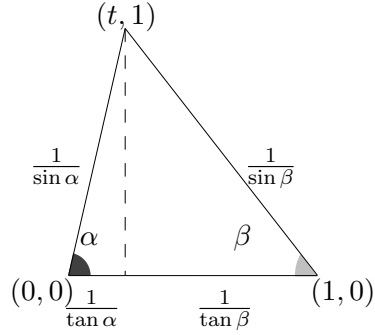}
    \caption{Visual aid for the proof of \eqref{claimdelta} for $M=3$.}
    \label{fig:stima}
\end{figure}
        Then, (see Figure \ref{fig:stima})
        \[
        |[e_2]|=\frac{1}{\sin\alpha}\qquad\textrm{and}\qquad |[e_3]|=\frac{1}{\sin\beta}.
        \]
        Also
        \begin{equation}\label{eq:1_over_tan}
        \frac{1}{\tan\beta}=1-\frac{1}{\tan\alpha}.
        \end{equation}
        We have that
        \[
        \phi(\nu_{e_2})=3\cos\big(\frac{\pi}{2}-\alpha\big)+\sin\big(\frac{\pi}{2}-\alpha\big)=3\sin\alpha+\cos\alpha.
        \]
        \[
        \phi(\nu_{e_3})=3\cos(\frac{\pi}{2}-\beta)+\sin(\frac{\pi}{2}-\beta)=3\sin\beta+\cos\beta.
        \]
        Using that $\ell_\phi([e]):=|[e]|\phi(\nu_e)$ we deduce
        \[
        \ell_\phi([e_2])=\frac{1}{\sin\alpha}(3\sin\alpha+\cos\alpha)=3+\frac{1}{\tan\alpha},
        \]
        and similarly
        \[
        \ell_\phi([e_3])=3+\frac{1}{\tan\beta}.
        \]
        We conclude that
        \begin{align*}
            \Per_\phi(F)&=\ell_\phi([e_1])+\ell_\phi([e_2])+\ell_\phi([e_3])\\
            &=3+3+\frac{1}{\tan\alpha}+3+\frac{1}{\tan\beta}\\
            &=10
        \end{align*}
        by \eqref{eq:1_over_tan}. This concludes the proof in the case $M=3$. 
        \item Finally, if $M=2$ then - adopting the notation in the previous case (see Figure \ref{fig:stima}) - either at least one of the two selected bonds is orthogonal, and the conclusion follows using that $3+\frac{1}{\tan \alpha}\le 4$; or instead both sides are diagonal, and then from the above computation $\ell_\phi([e_2])+\ell_\phi([e_3])=7$, which implies the thesis. \qedhere
    \end{itemize}
\end{proof}
Using the previous lemma on all the faces in $S$, we can finally prove that the term with $\Per_\comb(A_S(X))$ and the sum of the defects of the faces bound from above the anisotropic perimeter $\Per_\phi(A_S(X))$. 
\begin{lemma}\label{lemma:delta_versus_P_global}
    Let $X\in\A$. Let $S$ be a family of faces of $\G(X)$ with $\Fsf_\boxtimes(X)\subseteq S\subseteq \Fsf_\bdd(X)$. Then
    \[
    3\Per_\comb(A_S(X))+\sum_{F\in S} \delta(F)\ge \Per_\phi(A_S(X)).
    \]
\end{lemma}
\begin{proof}
    For every face $F$, let us call $\Ed^\partial_S(F)$ the family of the edges $e$ of $F$ that contribute to the perimeter of $A_S(X)$, namely, those satisfying $[e]\subseteq \partial A_S(X)$. We apply Lemma \ref{lemma:delta_versus_P} to every face $F\in S$ with $\mathsf{E}=\Ed^\partial_S(F)$, and to the edges in $\Ed_S^\partial(F)$:
    \begin{align*}
    3\Per_\comb(A_S(X))+\sum_{F\in S}\delta(F)&\ge\sum_{F\in S} \Big(3\# \Ed^\partial_S(F)+\delta(F)\Big)\\
    & \ge \sum_{F\in S} \sum_{[e]\in \Ed^\partial_S(F)} \ell_\phi([e])\\
    &=\Per_\phi(A_S(X)).\qedhere
    \end{align*}
\end{proof}
We are finally able to prove the $\Gamma$-$\liminf$ inequality.

\begin{proof}[Proof of Theorem \ref{thm:gamma_convergence}(i)]
We can assume without loss of generality that \eqref{bounden} holds true.
    Fix some $\beta\in (0,\tfrac12)$. Define $S_N$ to be the family of all crystallized square faces together with all the faces $F\in\Fsf(X_N)$ such that $\Per(F)\le N^\beta$. Accordingly, let 
    \[
    \widetilde A_N:=\bigcup_{F\in S_N} F.
    \]
    First we claim that 
    \begin{equation}\label{eq:tilde_A_N_converge_to_E}
    N^{-1/2} \widetilde A_N\to E \qquad\text{locally in measure.}
    \end{equation}
    Indeed, by the isoperimetric inequality we have that $|F|\leq c \Per^2(F)$ (for some constant $c$ independent of $F$), which, together with \eqref{bounden}, implies the following estimate:
    \begin{align*}
    |\widetilde A_N\setminus A_\boxtimes(X_N)|&=\sum_{\substack{F:\Per(F)\le N^\beta\\ F\not\in \Fsf_\boxtimes(X)}} |F|
    \le \sum_{\substack{F:\Per(F)\le N^\beta\\ F\not\in \Fsf_\boxtimes(X)}} c \Per^2(F)\\
    &\le \sum_{\substack{F:\Per(F)\le N^\beta\\ F\not\in \Fsf_\boxtimes(X)}} c \Per(F) N^\beta\\
    &\le c N^\beta \sum_{\substack{F:\Per(F)\le N^\beta\\ F\not\in \Fsf_\boxtimes(X)}} \sum_{v\in \mathrm{clos}(F)\cap X} \F(v)\\
    &\le c8N^\beta \F_N(X_N)\le c 8C N^{\beta+1/2},
    \end{align*}
    and by the choice $\beta<\tfrac12$ this implies
    \[
    |(N^{-1/2}\widetilde A_N)\setminus (N^{-1/2}A_\boxtimes(X_N))|\lesssim N^{\beta-1/2}\to 0,
    \]
    which proves \eqref{eq:tilde_A_N_converge_to_E}.

    Next we claim that 
    \begin{equation}\label{eq:chebishev_bound_cardinality}
        \# \{F\in \Fsf(X_N)\,:\, \Per(F)> N^\beta\}\le 8C N^{1/2-\beta}.
    \end{equation}
    This follows from the following Chebishev-type estimate:
    \begin{align*}
        \#\{F:\, \Per(F)> N^\beta\} N^\beta &\le \sum_{\{F:\, \Per(F)> N^\beta\}} \Per(F)\\
        & \le \sum_{\{F:\, \Per(F)> N^\beta\}} \sum_{v\in \mathrm{clos}(F)\cap X} \F(v)\\
        &\le 8 \F_N(X_N)\le 8C N^{1/2}.
    \end{align*}
Now we apply the energy decomposition given by Theorem \ref{thm:decomposition_square}, with $S=S_N$, to obtain that 
   \begin{align*}
    \F_N(X_N)&\ge 3\Per(\widetilde A_N)-8\#(\Fsf(X)\setminus S_N)+\sum_{F\in S_N}\delta(F)\\
    & \ge 3 \Per(\widetilde A_N) -64 C N^{1/2-\beta}+\sum_{F\in S_N}\delta(F)\\
    &\ge \Per_\phi(\widetilde A_N) -64 C N^{1/2-\beta}.
    \end{align*} 
    The last inequality follows from Lemma \ref{lemma:delta_versus_P_global}. In particular, by a rescaling we obtain
    \begin{equation}\label{eq:F_versus_P_phi}
        \Per_\phi\big(N^{-1/2}\widetilde A_N\big)\le N^{-1/2}\F_N(X_N)+64 C N^{-\beta}.
    \end{equation}
    By the lower semicontinuity of the anisotropic perimeter, \eqref{eq:tilde_A_N_converge_to_E} and \eqref{eq:F_versus_P_phi} we finally obtain
    \[
    \Per_\phi(E)\le \liminf_{N\to\infty} \Per_\phi\big(N^{-1/2}\widetilde A_N\big)\le \liminf_{N\to\infty} N^{-1/2}\F_N(X_N).
    \]
    This concludes the proof.
\end{proof}

We conclude this section by providing (a sketch of) the proof of the upper bound in Theorem \ref{thm:gamma_convergence}.

\begin{proof}[Proof of Theorem \ref{thm:gamma_convergence}(ii)]
By density in perimeter of polygonal sets in the space of finite perimeter sets, it is sufficient to construct a recovery sequence when the set is a polygon $E$ with area 1. In this case, the procedure is quite standard (see, for instance, \cite{gelli} and \cite{AuYeungFrieseckeSchmidt}): $X_N$ is obtained by considering $\widetilde X_N:=N^{1/2}E\cap \Z^2$, and by adjusting the number of particles by removing or adding some to reach cardinality $N$. This procedure has been performed many times
, and we report here only the computation to show that 
\begin{equation}\label{eq:limsup_easy}
\Per_\phi(E)\ge\limsup_{N\to\infty}N^{-1/2}\F_N(\widetilde X_N).
\end{equation}
It is not difficult to show that $\#\widetilde X_N=N+\mathrm{o}(N)$, and so to adjust the cardinality one can just add (or subtract) an ``almost square'' patch of particles, whose perimeter contribution will be $\mathrm{o}(\sqrt{N})$ and therefore negligible in the limit.

In order to prove \eqref{eq:limsup_easy}, we can reduce ourselves to compute the energy density per unit length of a single boundary edge with normal $\nu$. The idea is to further split the missing-bond energy contributions depending on the direction of the missing bonds: horizontal, vertical, diagonal NE-SW or diagonal NW-SE. We refer also to Fig. \ref{fig:recovery_sequence}, where the missing bonds have been colored in red and split according to their direction. Over a length $L$, the number of horizontal bonds coincides with $|L\cos\alpha|=L|\nu_1|$, where $\alpha$ is the angle between the side and the vertical axis or, equivalently, between the normal $\nu$ and the horizontal axis. Therefore the number of horizontal bonds per unit length is $|\nu_1|$. Similarly, the number of vertical bonds per unit length is $|\nu_2|$. Concerning the NW-SE diagonal bonds, there is one for every boundary point that is ``directly visible from west'' (density $-\nu_1$), plus another one for every concave corner, or equivalently, for every point ``directly visible from north'' (density $\nu_2$). In conclusion, the density in this case is $|\nu_1-\nu_2|$. The last case is similar, giving density $|\nu_1+\nu_2|$. By summing all these contributions we get precisely $\phi(\nu)$, which is what we wanted.
\begin{figure}[htbp]
\begin{minipage}{0.1\textwidth}
    \def\svgwidth{0.35\columnwidth}
\begingroup%
  \makeatletter%
  \providecommand\color[2][]{%
    \errmessage{(Inkscape) Color is used for the text in Inkscape, but the package 'color.sty' is not loaded}%
    \renewcommand\color[2][]{}%
  }%
  \providecommand\transparent[1]{%
    \errmessage{(Inkscape) Transparency is used (non-zero) for the text in Inkscape, but the package 'transparent.sty' is not loaded}%
    \renewcommand\transparent[1]{}%
  }%
  \providecommand\rotatebox[2]{#2}%
  \newcommand*\fsize{\dimexpr\f@size pt\relax}%
  \newcommand*\lineheight[1]{\fontsize{\fsize}{#1\fsize}\selectfont}%
  \ifx\svgwidth\undefined%
    \setlength{\unitlength}{550.97088671bp}%
    \ifx\svgscale\undefined%
      \relax%
    \else%
      \setlength{\unitlength}{\unitlength * \real{\svgscale}}%
    \fi%
  \else%
    \setlength{\unitlength}{\svgwidth}%
  \fi%
  \global\let\svgwidth\undefined%
  \global\let\svgscale\undefined%
  \makeatother%
  \begin{picture}(1,0.68583332)%
    \lineheight{1}%
    \setlength\tabcolsep{0pt}%
    \put(0,0){\includegraphics[width=\unitlength,page=1]{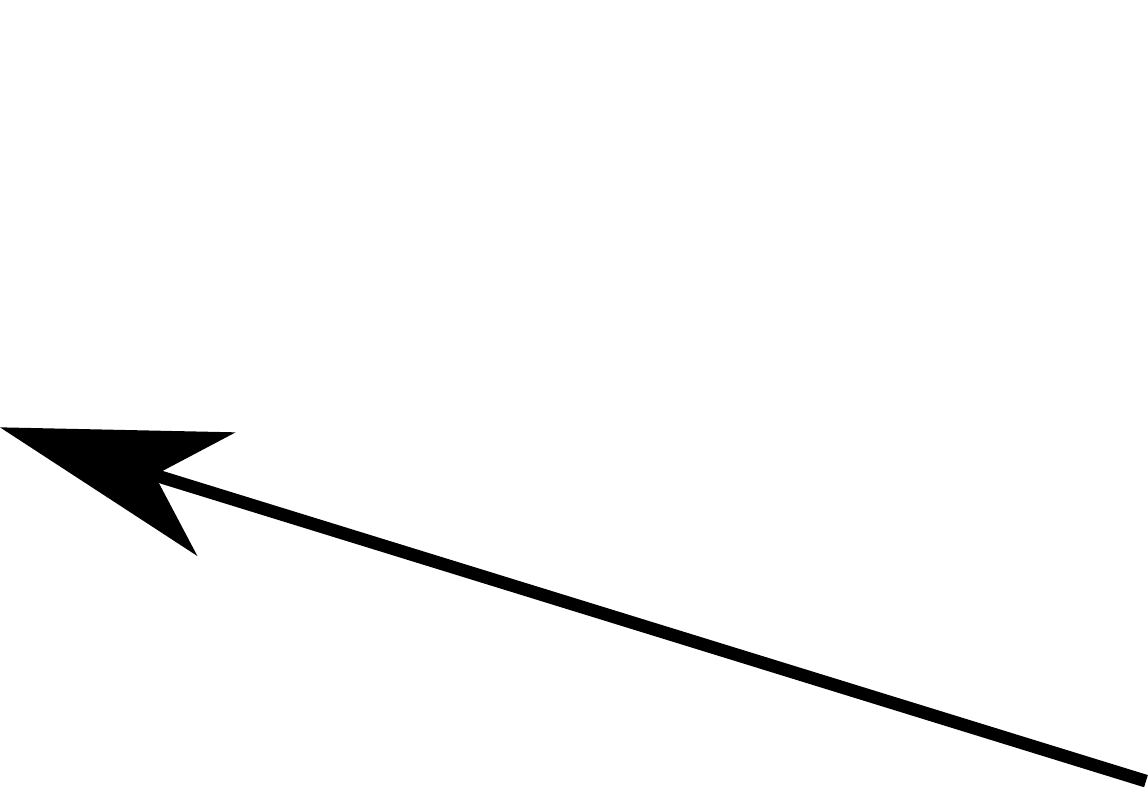}}%
    \put(0.24435079,0.50606539){\makebox(0,0)[lt]{\lineheight{1.25}\smash{\begin{tabular}[t]{l}$\nu$\end{tabular}}}}%
  \end{picture}%
\endgroup%

\end{minipage}
\begin{minipage}{0.2\textwidth}
\includegraphics[width=0.9\textwidth]{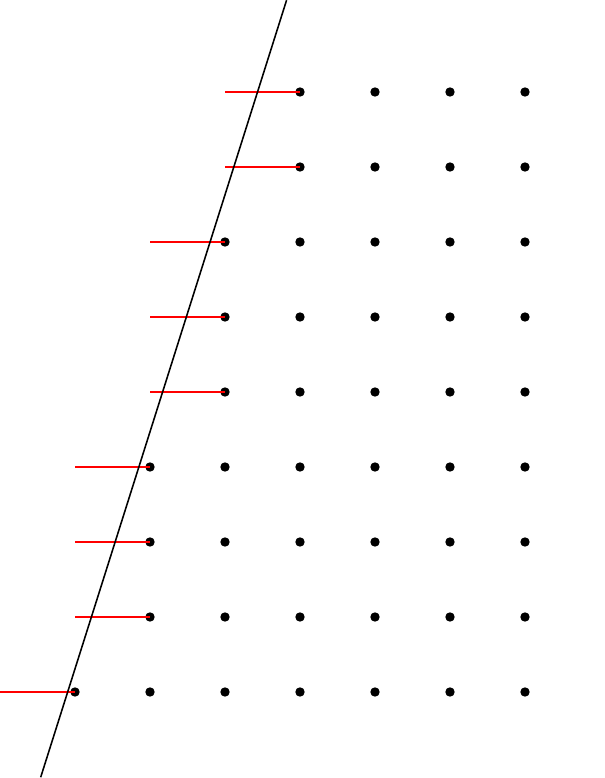}
\end{minipage}
\begin{minipage}{0.2\textwidth}
\includegraphics[width=0.9\textwidth]{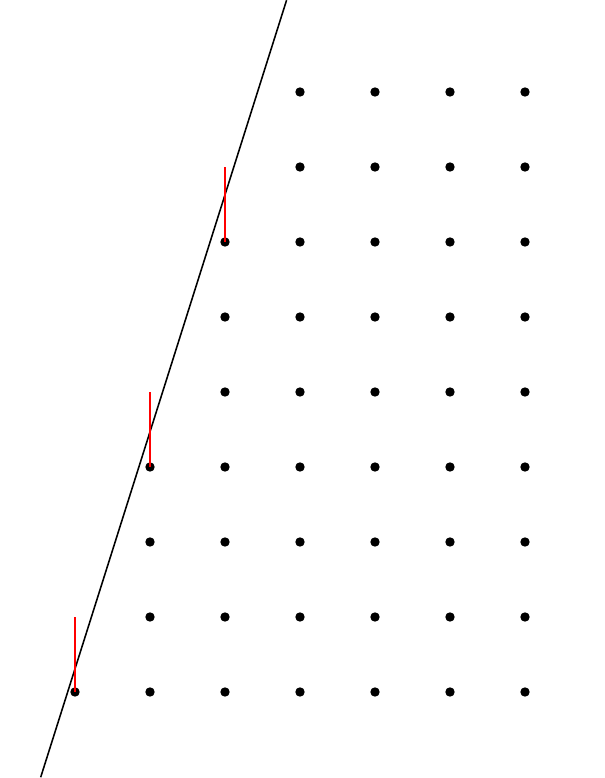}
\end{minipage}
\begin{minipage}{0.2\textwidth}
\includegraphics[width=0.9\textwidth]{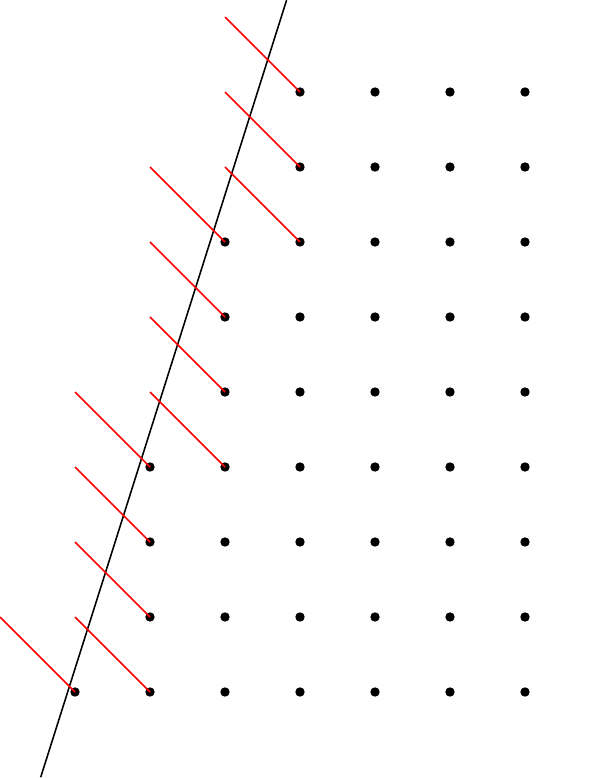}
\end{minipage}
\begin{minipage}{0.2\textwidth}
\includegraphics[width=0.9\textwidth]{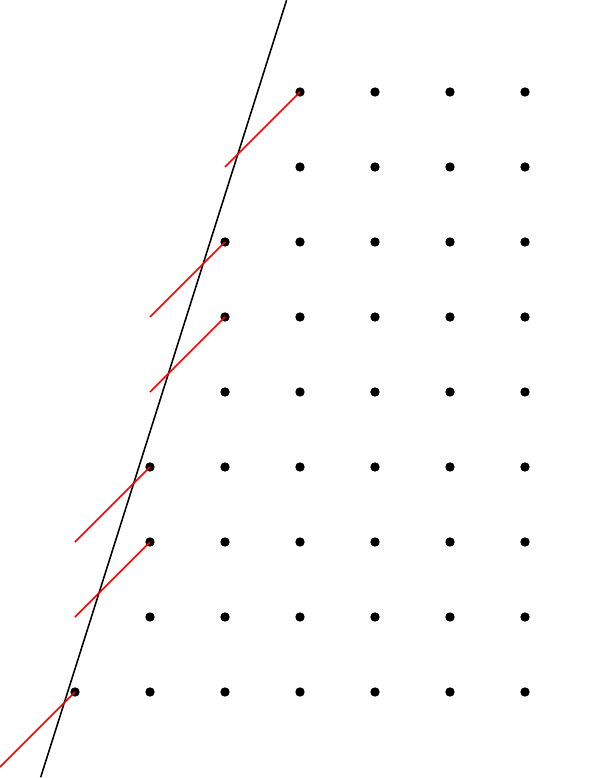}
\end{minipage}    
    \caption{To compute the excess energy $\F$ due to missing bonds it is useful to view it as a superposition of different directional contributions: it is elementary to compute that the number of missing bonds per unit length in the four cases shown above is, respectively, $|\nu_1|$, $|\nu_2|$, $|\nu_1-\nu_2|$, $|\nu_1+\nu_2|$. The sum of these contributions gives the anisotropy $\phi$.}
    \label{fig:recovery_sequence}
\end{figure}
\end{proof}

\begin{remark}
    We finally observe that the Wulff shape associated with $\Per_\phi$ is an octagon whose sides have unit length with respect to $\lVert\cdot\rVert_\infty$, and hence their Euclidean length is 1 for horizontal and vertical sides, and $\sqrt{2}$ for diagonal sides. To see this, besides computing explicitly the polar body, we note that there is a simple way to find out Wulff shapes of anisotropies that are the sum of terms of the form $|\nu\cdot v_i|$, for some $v_i$:  one can just take the Minkowski sum of the $v_i$'s. We refer to, e.g., \cite[Corollary~1.3]{DelNinPetrache} for a similar result in a periodic and quasiperiodic setting. 
\end{remark}

\end{document}